\documentclass[reqno,a4paper]{amsart}  
\usepackage{enumerate,amsmath,amssymb,amscd}     
\usepackage{hyperref}
\usepackage{MnSymbol}
\usepackage{color}

\theoremstyle{plain}      
 
\newtheorem{thm}{Theorem}[section]     
     
\newtheorem{cor}[thm]{Corollary}     
     
\newtheorem{lem}[thm]{Lemma}     
\newtheorem{lemma}[thm]{Lemma}     
\newtheorem{prop}[thm]{Proposition}

\theoremstyle{remark}       
\newtheorem*{example*}{Example}
 
\theoremstyle{definition}      
\newtheorem{rem*}[thm]{Remark} 

\newtheorem{defn}[thm]{Definition}     
\newtheorem{definition}[thm]{Definition}     

\def\Si{{\Sigma}}

\def\epsilon{{\varepsilon}}

\def\phi{{\varphi}}
\def\Ph{{\Phi}}

\def\cliff{\mathop{\mathrm{Cliff}}}
\def\diff{\mbox{\sl Diff}}

\DeclareMathAlphabet{\doba}{U}{msb}{m}{n}

\gdef\mZ{\doba{Z}}

\def\cE{\mathcal{E}}

\def\cS{\mathcal{S}}

\def\End{{\mathop\mr{End}}}

\def\vol{{\mathop{\rm vol}}}

\def\Ric{{\mathop{\rm Ric}}}

\def\tr{{\mathop{\rm Tr}}}
\def\note#1{\marginpar{\raggedright\if@twoside\ifodd\c@page\raggedleft\fi\fi\sf\scriptsize RMK: #1}}

\let\scal\Scal
     
\let\na\nabla     
\let\<\langle 
\let\>\rangle 
\let\ti\tilde

\newcommand{\definedas}{\mathrel{\raise.095ex\hbox{\rm :}\mkern-5.2mu=}}
     
\newcommand{\ind}{\mr{ind}\,}

\newcommand{\SU}{\mr{SU}}
\newcommand{\SO}{\mr{SO}}
\newcommand{\Sp}{\mr{Sp}}
\newcommand{\GL}{\mr{GL}}

\newcommand{\Gt}{\mr{G_2}}
\newcommand{\Spin}{\mr{Spin}}

\newcommand{\id}{\text{Id}}

\newcommand{\mr}{\mathrm}

\newcommand{\mc}{\mathcal}

\newcommand{\R}{\mathbb{R}}
\newcommand{\Z}{\mathbb{Z}}
\newcommand{\C}{\mathbb{C}}
\newcommand{\Qa}{\mathbb{H}}

\renewcommand{\iff}{if and only if }

\newcommand{\st}{such that }

\newcommand{\Id}{\operatorname{id}}

\newcommand{\beq}{\begin{equation}}
\newcommand{\ben}{\begin{equation}\nonumber}
\newcommand{\ee}{\end{equation}}

\newcommand{\grad}{\mr{grad}\,}

\renewcommand{\div}{\operatorname{div}}
\newcommand{\im}{\operatorname{im}}
\newcommand{\sgn}{\operatorname{sgn}}

\newenvironment{thm*}{\begin{trivlist}\item[]{\bf Theorem}}{\end{trivlist}}
\newenvironment{dfn*}{\begin{trivlist}\item[]{\bf Definition.}}{\end{trivlist}}
\newenvironment{prp*}{\begin{trivlist}\item[]{\bf Proposition}}{\end{trivlist}}
\newenvironment{lem*}{\begin{trivlist}\item[]{\bf Lemma.}}{\end{trivlist}}
\newenvironment{cor*}{\begin{trivlist}\item[]{\bf Corollary.}}{\end{trivlist}}

\begin{document}     
%
%
\title{A spinorial energy functional: critical points and gradient flow}
 
\author{Bernd Ammann} 
\address{Fakult\"at f\"ur Mathematik \\
Universit\"at Regensburg \\ Universit\"atsstra{\ss}e 40 \\
D--93040 Regensburg \\  
Germany}
\email{bernd.ammann@mathematik.uni-regensburg.de}

\author{Hartmut Wei\ss{}} 
\address{Mathematisches Seminar der Universit\"at Kiel\\ Ludewig-Meyn Stra{\ss}e 4\\ D--24098 Kiel\\ Germany}
\email{weiss@math.uni-kiel.de}

\author{Frederik Witt} 
\address{Institut f\"ur Geometrie und Topologie der Universit\"at Stuttgart\\ Pfaffenwaldring 57\\ D--70569 Stuttgart\\ Germany}
\email{frederik.witt@mathematik.uni-stuttgart.de}

\begin{abstract}
Let $M$ be a compact spin manifold. On the universal bundle of unit spinors we study a natural energy functional whose critical points, if $\dim M \geq 3$, are precisely the pairs $(g,\phi)$ consisting of a Ricci-flat Riemannian metric~$g$ together with a parallel $g$-spinor~$\phi$. We investigate the basic properties of this functional and study its negative gradient flow, the so-called spinor flow. In particular, we prove short-time existence and uniqueness for this flow.
\end{abstract}

\subjclass[2000]{}



\maketitle

%
\section{Introduction}
%
%
Approaching special holonomy metrics and related structures from a spinorial point of view often gives interesting insights. For example, one can characterise Ricci-flat metrics of special holonomy in terms of parallel spinors~\cite{hi74},~\cite{wa89}. This enabled Wang to show that if $(M,g)$ is a simply-connected, irreducible Riemannian manifold with a parallel spinor, any metric in a suitable Einstein neighbourhood of that metric must also admit a parallel spinor and is thus of special holonomy as well, see~\cite{wa91}. More generally, one can consider Killing spinors (in the sense of~\cite{bfgk91}) which forces the underlying metric $g$ to be Einstein. In fact, they arise in connection with Einstein-Sasaki structures (see for instance~\cite{bfgk91},~\cite{boga99}) and also relate to Gray's notion of {\em weak} holonomy groups~\cite{gr71}. B\"ar's classification~\cite{ba93} links Killing spinors to parallel spinors on the cone $M\times\R_+$ with warped product metric $t^2g+dt^2$, thereby relating these Einstein geometries to special holonomy metrics. More generally still, one can consider generalised Killing spinors in the sense of~\cite{bgm05} in connection with embedding problems. In~\cite{amm11} it is shown that a manifold of arbitrary dimension $n$ which carries a real-analytic generalised Killing spinor embeds isometrically into a manifold of dimension $n+1$ with a parallel spinor. This generalises results on Hitchin's embedding problem~\cite{hi01} in dimension $n=7$, $n=6$ and $n=5$ with co-calibrated $G_2$-, half-flat and hypo-structures respectively~\cite{cosa06} to arbitrary dimensions. 

\medskip

The present article is the first one of a programme to understand and solve these spinor field equations from a variational point of view. This also puts a previous result of the last two authors on a certain heat flow on $7$-dimensional manifolds~\cite{ww10} into a framework which is valid in any dimension. More concretely, let $M$ be an $n$-dimensional, compact, oriented spin manifold with a given spin structure, and consider the universal bundle of unit spinors $S(\Sigma M)\to M$. A section $\Phi\in \mc N := \Gamma(S(\Sigma M))$ can be regarded as a pair $(g,\phi)$ where $g$ is a Riemannian metric and $\phi\in\Gamma(\Sigma_gM)$ is a $g$-spinor of constant length one. We then introduce the energy functional $\mc{E}$ defined by
\ben
\mc{E}:\mc{N}\to\R_{\geq 0},\quad \Phi \mapsto\tfrac{1}{2}\int_M|\nabla^g\phi|_g^2\,dv^g,
\ee
where $\nabla^g$ denotes the Levi-Civita connection on the $g$-spinor bundle $\Sigma_gM$, $|\,\cdot\,|_g$ the pointwise norm on $T^*\!M \otimes \Sigma_gM$ and $dv^g$ the Riemann-Lebesgue measure given by the volume form of $g$. If $M$ is a compact Riemann surface, then a conformal immersion $F:M \to \R^3$ induces a spin structure on $M$ and a section $(g,\phi)\in \mc N$ on $M$. The spinorial version of the Weierstrass representation (see~\cite{fr98}), implies $D_g\phi= H\phi$ where $H$ is the mean curvature function of the immersed surface. In this case we obtain $2\mc E(g,\phi)= \int_M H^2\, dv^g$ which is the Willmore energy of the immersion. Our functional thus extends the Willmore functional to a larger domain. We elaborate on this point in~\cite{awwII}.

\medskip

Our first result characterises the (unconstrained) critical points.

\smallskip

\begin{prp*} {\bf A (Critical points).}
{\em Let $n\geq3$. Then $(g,\phi)\in\mc N$ is critical if and only if it is an absolute minimiser, i.e.\ $\nabla^g\phi=0$. 
In particular, the metric $g$ is Ricci-flat and of special holonomy.}
\end{prp*}

\smallskip

The case $\dim M=2$ is of a different flavour and is be also treated in detail in~\cite{awwII}. There, we have a trichotomy for the absolute minimisers according to the genus $\gamma$, namely twistor spinors for $\gamma=0$, parallel spinors for $\gamma=1$ and harmonic spinors for $\gamma\geq2$. However, non-minimal critical points do exist for $\gamma\geq1$. Coming back to the case of $n\geq3$, we can also consider critical points subject to the constraint $\mr{Vol}(g)=\int_M dv^g=1$. Particular solutions are given by Killing spinors (as above in the sense of~\cite{bfgk91}). We expect more general solutions to exist, but we will leave this, as well as a systematic investigation of the resulting ``soliton equation'' following the lines of~\cite{ww12}, to a further paper. Proposition~A as well as the related results we have just                                                                                                       mentioned follow from the computation of the negative $L^2$-gradient $Q:\mc N\to T\mc N$ of $\mc E$. The main technical ingredient is the Bourguignon--Gauduchon ``partial connection'' on the fiber bundle $S(\Sigma M)$ which yields a horizontal distribution on the (Fr\'echet) vector bundle $\mc N\to\mc M$, where $\mc M$ is the space of Riemannian metrics on~$M$, cf.~\cite{bg91}. In $T_{(g,\phi)}\mc N$ the vertical space is $\Gamma(\phi^\perp)$, the space of spinors pointwise orthogonal to $\phi$. The Bourguignon--Gauduchon horizontal distribution provides a natural complement, isomorphic to $T_g\mc M$, the space of symmetric $2$-forms on $M$. This formalism also underlies Wang's pioneering work on the deformation of parallel spinors under variation of the metric~\cite{wa91} mentioned above. However, instead of using the universal spinor bundle, he considers a fixed spinor bundle where the variation of the metric materialises as a variation of the connections with respect to which one differentiates the spinor fields. Bearing this in mind, some of our formul{\ae} already appear in~\cite{wa91}, see Remark~\ref{wang}.

\medskip

In order to detect critical points it is natural to consider the negative gradient flow
\beq\label{floweqintro}
\tfrac{\partial}{\partial t}\Phi_t=Q (\Phi_t),\quad\Phi_0=\Phi
\ee
for an arbitrary initial condition $\Phi=(g,\phi)\in\mc N$. One may wonder if any favourable behaviour of this flow is to be expected from convexity properties of the functional. However, the functional $\mc E$ is convex only under variation of the spinor field $\varphi$ alone (keeping the metric $g$ fixed). In general, under variation of both the spinor field and the metric, convexity fails as is shown by the existence of critical points on the 2-torus with indefinite Hessian~\cite{awwII}.
Nevertheless, and somewhat suprisingly, the second order partial differential operator defined by linearising the gradient of the energy  has positive definite principal symbol up to diffeomorphisms. From there we obtain the

\smallskip

\begin{thm*} {\bf B (Short-time existence and uniqueness).}
{\em For all $\Phi\in\mc N$, there exists $\epsilon>0$ and a smooth family 
$\Phi_t\in\mc N$ for $t\in[0,\epsilon]$ such that~\eqref{floweqintro} holds. 
Furthermore, if $\Phi_t$ and $\Phi'_t$ are solutions to~\eqref{floweqintro}, 
then $\Phi_t=\Phi'_t$ whenever defined. Hence $\Phi_t$ is uniquely defined on 
a maximal time-interval $[0,T)$ for some $0<T\leq\infty$.
}
\end{thm*}

\smallskip

The group of spin-diffeomorphisms $\widehat\diff_s(M)$ acts on pairs $(g,\phi)$ in a natural way and contains in particular the universal covering group $\widetilde\diff_0(M)$ of the group of diffeomorphisms isotopic to the identity. Since the operator $Q$ is equivariant with respect to this action, its linearisation has an infinite-dimensional kernel. 
The principal symbol of the linearisation is positive semi-definite. 
A key observation is that the kernel of the principal symbol is solely due to this diffeomorphism invariance.
Though the standard theory of (strongly) parabolic evolution equations does not apply directly anymore, $Q$ is still weakly parabolic. Using a variant of DeTurck's trick~\cite{dt83} as in~\cite{ww10}, we can prove Theorem~B for the perturbed flow equation with
\ben
\widetilde Q_{\Phi_0}:\mc N\to T\mc N,\quad\Phi\mapsto Q(\Phi)+\lambda_{\Phi}^*(X_{\Phi_0}(\Phi))
\ee
instead of $Q$. Here, $X_{\Phi_0}$ is a certain vector field which depends on the $1$-jet of the initial condition $\Phi_0$. Further, $\lambda^*_{\Phi}(X)= (2 \delta_g^* X^\flat, \nabla^g_X\phi-\tfrac{1}{4} dX^\flat \cdot \phi)$ where $\Phi=(g,\phi)$ and $\delta_g^*$ is the adjoint of the divergence operator associated with $g$. The second component of $\lambda^*_\Phi$ is the spinorial Lie derivative as defined in~\cite{bg91} and~\cite{ko72}. Roughly speaking the degeneracy of $Q$ is eliminated by breaking the $\diff_0(M)$-equivariance with an additional Lie derivative term. One can then revert solutions for the perturbed flow back into solutions of~\eqref{floweqintro}. 

\medskip

The perturbed operator $\tilde Q_{\Phi_0}$ also appears in connection with the premoduli space of critical points. The set of critical points of $\mc E$, $\mr{Crit(\mc E)}$, fibres over the subset of Ricci-flat metrics with a parallel spinor, the fibres being the finite-dimensional vector spaces of parallel spinors. In general, the dimension of these spaces need not to be locally constant. This, however, will be the case if $M$ is simply-connected and $g$ irreducible as a consequence of Wang's stability~\cite{wa91} and Goto's unobstructedness~\cite{go04} theorem. Recall that a function is said to be Morse-Bott if its critical set is smooth and if it is non-degenerate transverse to the critical set. We get

\smallskip

\begin{thm*} {\bf C (Smoothness of the critical set).}
{\em Let $M$ be simply-connected and $\bar\Phi=(\bar g,\bar\phi)\in\mc N$ be an irreducible critical point, i.e.\ its underlying metric is irreducible. Then $\mr{Crit(\mc E)}$ is smooth at $\bar\Phi$, i.e. its Zariski tangent space is integrable. Further, $\tilde Q^{-1}_{\bar\Phi}(0)$ is a slice for the $\widetilde\diff_0(M)$-action on $\mr{Crit(\mc E)}$, that is, the premoduli space of parallel spinors $\mr{Crit(\mc E)}/\widetilde\diff_0(M)$ is smooth at $\bar\Phi$. If all critical points are irreducible, then $\mc E$ is Morse-Bott.}
\end{thm*}

Note that under the assumption that $M$ is simply-connected any critical point is irreducible in dimensions $4$, $6$ and $7$. The same holds in dimension $8$ unless $M$ is a product of two $K3$-surfaces. Theorem C holds more generally on certain non simply-connected manifolds, see Theorem~\ref{integkerL}. In fact, ongoing work~\cite{holrig} will actually show that one can remove the assumptions ``simply-connected'' and ``irreducible'' so that $\mr{Crit}(\mc E)$ is always smooth, cf.\ also Remark~\ref{art.in.pre}.
\medskip

Theorem C and the formula for the second variation of $\mc E$ which we compute in Section~\ref{second_var} give all the necessary ingredients for stability of the flow in the sense of~\cite[Theorem 8.1]{ww10}. Namely, in a suitable $C^\infty$-neighbourhood of an irreducible critical point, we expect the flow~\eqref{floweqintro} to exist for all times and to converge modulo diffeomorphisms to a critical point. This is now work in progress by Lothar Schiemanowski as part of his thesis project. Ideally, the flow could become a tool for detecting special holonomy metrics, metrics with (generalised) Killing spinors, twistor spinors and other solutions to ``natural'' spinor field equations. Also the limit spaces when solutions of the flow develop singularities should be of interest and we hope to develop these issues further in the near future. 

\medskip

\paragraph{\bf Acknowledgements.} The authors thank the referees for carefully reading the manuscript which led to considerable improvements of the text.

%
%
\section{Spin geometry}\label{spingeo}
%
%
In this section we set up our conventions relevant for the subsequent computations and recall the basic spin geometric definitions. Suitable references for this material are~\cite{fr00} and~\cite{lami89}. 

\medskip  

Throughout this paper, $M^n$ will denote a connected, compact, oriented smooth manifold of dimension $n\geq2$. We write $\mc M$ for the space of Riemannian metrics on~$M$. The choice of $g\in \mc M$ gives us the Riemannian volume form $\vol_g$ and allows us to identify vectors with covectors via the musical isomorphisms ${}^\flat: TM \rightarrow T^*\!M$ and ${}^\sharp: T^*\!M \rightarrow TM$. We will often drop the disctinction between $v$ and $v^\flat$ for $v \in TM$ (resp.\ between $\xi$ and $\xi^\sharp$ for $\xi \in T^*\!M$). In particular,  we will identify a local orthonormal frame $e_1, \ldots , e_n$ with its dual coframe. Further, the metric $g$ induces bundle metrics on the tensor powers $\bigotimes^p T^*\!M$ and the exterior powers $\Lambda^pT^*\!M$: If $e_1,\ldots,e_n$ is a local orthonormal frame for $TM$, then $e_{i_1}\otimes\ldots\otimes e_{i_p}$, $1 \leq i_1, \ldots, i_p \leq n$ is a local orthonormal frame for $\bigotimes^pT^*\!M$ and $e_{1_1} \wedge \ldots \wedge e_{i_p}$, $1 \leq i_1 < \ldots < i_p \leq n$ for $\Lambda^pT^*\!M$. We embed $\Lambda^p T^*\!M$ into $\bigotimes^p T^*\!M$ via
\[
\xi_1 \wedge \ldots \wedge \xi_n \mapsto \sum_{\sigma\in S_p} \sgn \sigma \cdot \xi_{\sigma(1)} \otimes \ldots \otimes \xi_{\sigma(p)},
\]
which, however, is not an isometric embedding. For example, considered as an element in $\Lambda^nT^*\!M$, the volume form is given by $\vol_g =e_1 \wedge \ldots \wedge e_n$ with respect to a local orthonormal frame and has unit length. Considered as an anti-symmetric element in $\bigotimes^nT^*\!M$ it has length $\sqrt{n!}$ and satisfies $\vol_g(e_1, \ldots, e_n)=1$. In the symmetric case, we only consider $\bigodot^2 T^*\!M$ which we view as a subspace of $\bigotimes^2T^*\!M$ via the embedding
\ben
\xi_1 \odot \xi_2 \mapsto \tfrac{1}{2} (\xi_1 \otimes \xi_2 + \xi_2 \otimes \xi_1).
\ee
We then equip $\bigodot^2T^*\!M$ with the induced metric. Using this convention, the metric may be expressed as $g= \sum_{i=1}^n e_i \odot e_i = \sum_{i=1}^n e_i \otimes e_i$ with respect to a local orthonormal frame and has length $\sqrt{n}$. In all cases we generically write $(\cdot\,,\cdot)_g$ for metrics on tensors associated with $g$ in this way, and $|\cdot|_g$ for the associated norm. Finally, we denote the associated Riemann--Lebesgue measure by $dv^g$. This yields an $L^2$-inner product on any such bundle via
\ben
\llangle\alpha,\beta\rrangle_g=\int_M(\alpha,\beta)_g\,dv^g
\ee
for $\alpha,\,\beta\in\bigotimes^p T^*\!M$, resp.\ $\in \Lambda^pT^*\!M$. We write $\|\cdot\|_g$ for the associated $L^2$-norm.

\medskip

In the sequel we require $M$ to be a spin manifold. By definition, this means that for the principal $\GL^+_n$-bundle of oriented frames $P$ there exists a twofold covering by a principal $\widetilde\GL^+_n$-bundle $\tilde P$ which fiberwise restricts to the universal covering map $\theta:\widetilde\GL^+_n\to\GL^+_n$ for $n\geq3$ (resp.\ the connected double covering for $n=2$). We call $\tilde P\to P$ a {\em spin structure}. We always think of a spin manifold as being equipped with a fixed spin structure which in general is not unique. In fact, $H^1(M,\Z_2)$ acts freely and transitively on the set of equivalence classes of spin structures. The additional choice of $g\in\mc M$ reduces $P$ to the 
principal $\SO_n$-fiber bundle $P_g$ of oriented $g$-orthonormal frames. 
This in turn is covered by a uniquely determined principal $\Spin_n$-bundle 
$\tilde P_g$ which reduces $\tilde P$, where $\Spin_n$ is by definition the 
inverse image of $\SO_n$ under $\theta$, i.e.\ $\Spin_n=\theta^{-1}(SO_n)$.

\medskip

To introduce spinors we need to consider representations of $\Spin_n$.  For simplicity we will restrict the discussion in this article to complex representations unless specified otherwise. However, all results continue to hold if we replace complex representations by real representations taking into account the real representation theory, cf.\ for instance~\cite[Proposition I.5.12]{lami89} and the discussion in Section~\ref{pointrepsym}. If we view $\Spin_n$ as a subgroup of the group of invertible elements in $\cliff_n$, the Clifford algebra of Euclidean space $(\R^n,g_0)$, then every irreducible representation of $\cliff_n$ restricts to a representation of $\Spin_n$. As shown in \cite[I.5]{lami89} there is up to isomorphism only one such representation of $\Spin_n$ in any dimension, which is called  {\em the spin representation $\Si_n$}. Note that if $n$ is odd this representation extends to two non-isomorphic representations of the Clifford algebra which only differ by a sign. The non-trivial element in the kernel of $\Spin_n\to \SO_n$ acts as $-\Id$ in this spin representation. The elements of $\Sigma_n$ are referred to as {\em spinors}. The (complex) dimension of $\Sigma_n$ is $2^{[n/2]}$, where $k=[r]$ is the largest integer $k\leq r$. For $n$ even, $\Si_n$ is an irreducible representation of the Clifford algebra, but decomposes as a $\Spin_n$ representation into $\Si_n=\Si_n^+\oplus \Si_n^-$. For $n$ odd, $\Si_n$ is an irreducible representation of $\Spin_n$, which however extends to two non-equivalent representations of $\cliff_n$. An important feature of spinors is that they can be multiplied by vectors. More precisely, the above mentioned representation of the Clifford algebra defines a $\Spin_n$-equivariant bilinear map $\mu:\R^n\times\Sigma_n^{(\pm)}\to\Sigma_n^{(\mp)}$ called {\em Clifford multiplication}. We usually write $v\cdot\phi=\mu(v,\phi)$. This can be extended to forms as follows: If $\phi\in\Sigma_n$, $\alpha\in\Lambda^p\R^{n*}$ and $E_1,\ldots,E_n$ denotes the standard oriented orthonormal basis of $\R^n$, then
\ben
\alpha\cdot\varphi:=\sum\limits_{1\leq j_1<\ldots< j_p\leq n}\alpha(E_{j_1},\ldots,E_{j_p})E_{j_1}\cdot\ldots\cdot E_{j_p}\cdot\varphi,
\ee
which is again a $\Spin_n$-equivariant operation. In particular, we have
\ben
(X\wedge Y)\cdot\phi=X\cdot Y\cdot\phi+ g(X,Y)\phi.
\ee
Finally, there exists a $\Spin_n$-invariant hermitian inner product $h$ on $\Sigma_n$ which for us means in particular that $h$ is positive definite. This gives rise to a positive definite real inner product $\<\cdot\,,\cdot\>=\mr{Re}\,h$ for which Clifford multiplication is skew-adjoint, that is, $\<v\cdot\phi,\psi\>=-\<\phi,v\cdot\psi\>$ for all $v\in\R^n$, $\phi,\,\psi\in\Sigma_n$. It follows that $\<\alpha\cdot\varphi,\psi\>=(-1)^{p(p+1)/2}\<\varphi,\alpha\cdot\psi\>$ for $\alpha\in\Lambda^p\R^{n*}$. 

\medskip

Coming back to the global situation, the choice of a metric enables us to define the vector bundle
\ben
\Sigma_gM=\tilde P_g\times_{\Spin_n}\!\Sigma_n
\ee
associated with $\tilde P_g$. By equivariance, Clifford multiplication, 
the hermitian inner product and the real inner product $\<\cdot\,,\cdot\>$ make 
global sense on $M$ and will be denoted by the same symbols. We denote by
\ben
\mc{F}_g = \Gamma(\Sigma_gM)\quad\mbox{and}\quad\mc{N}_g=\{\phi\in\mc F_g : |\phi|_g=1\}
\ee 
the space of (unit) sections, called {\em (unit) spinor fields}, or {\em (unit) spinors} for short. These spaces carry an $L^2$-inner product given by
\ben
\llangle\phi,\psi\rrangle_g=\int_M\langle\phi,\psi\rangle \,dv^g.
\ee
This inner product is again a real inner product. It is in fact the real part of an hermitian inner product, but for our purposes it is more convenient to work with the real inner product.
The Levi-Civita connection $\nabla^g$ can be lifted to a metric connection on $\Sigma_gM$ which we denote by $\nabla^g$ as well. 
For $X,\,Y\in\Gamma(TM)$ and $\phi\in\mc F_g$ it satisfies
\ben
\nabla^g_X(Y\cdot\phi)=(\nabla^g_XY)\cdot\phi+Y\cdot\nabla_X^g\phi.
\ee
In terms of a local representation $[\tilde b,\tilde\phi]$ of $\phi$, where for an open $U\subset M$, 
$\tilde\phi$ is a map $U\to\Sigma_n$ and $\tilde b:U\to\tilde P_g$ covers a local orthonormal 
basis $b=(e_1,\ldots,e_n):U\to P_g$, we have
\beq\label{spinor_conn}
\nabla^g_X\phi=[\tilde b,d\tilde\phi(X)+\tfrac{1}{2}\sum_{i<j}g(\nabla^g_Xe_i,e_j)E_i\cdot E_j\cdot\tilde\phi].
\ee
The action of the curvature operator associated with the Levi-Civita connection, $R^g(X,Y):TM\to TM$ for vector fields $X,\,Y\in\Gamma(TM)$, gives also rise to an action $\mc F_g\to\mc F_g$, namely
\beq\label{curvatureaction}
R^g(X,Y)\phi:=\nabla^g_X\nabla^g_Y\phi-\nabla^g_Y\nabla^g_X\phi-\nabla^g_{[X,Y]}\phi.
\ee
Expressed in a local orthonormal basis we have
\ben
R^g(X,Y)\phi:=\tfrac{1}{2}\sum_{j<k}g(R^g(X,Y)e_j,e_k)e_j\cdot e_k\cdot\phi
\ee
which yields using the first Bianchi identity
\ben
\sum_{i=1}^n e_i\cdot R^g(X,e_i)\cdot\phi= -\tfrac 12 \Ric (X)\cdot\phi.
\ee
In particular, if $g$ admits a {\em parallel} spinor, i.e.\ there is a $\phi\in\mc F_g\setminus\{0\}$ satisfying $\nabla^g\phi=0$, then $g$ is Ricci-flat.
%
%
\section{The dependence of the spinors on the metric}\label{dependence}
%
%
The very definition of spinors requires the a priori choice of a metric: Any finite-dimensional representation of $\widetilde\GL^+_n$ factorises via $\GL^+_n$, so that there are no spin representations for the general linear group. Consequently, any object involving spinors will in general depend on the metric. We therefore need a way to compare spinors defined with respect to different metrics, that is, we need some kind of connection.
%
\subsection{The universal spinor bundle}
%
As a bundle associated with the $\GL_n^+$-principal bundle $P$ of oriented frames, the bundle of positive definite bilinear forms is given by
\ben 
\odot_+^2T^*\!M = P\times_{\GL_n^+} \GL_n^+/\SO_n=P/\SO_n.
\ee
In particular, the projection $P \rightarrow \odot_+^2T^*\!M$ is a principal $\SO_n$-bundle. If a spin structure $\tilde P\rightarrow P$ is chosen, we may also write
\ben
\odot_+^2T^*\!M = \tilde P \times_{\widetilde\GL_n^+} \GL_n^+/\SO_n=\tilde P \times_{\widetilde\GL_n^+} \widetilde\GL_n^+/\Spin_n =\tilde P /\Spin_n
\ee
and the projection $\tilde P\rightarrow \odot_+^2T^*\!M$ becomes a $\Spin_n$-principal bundle. The {\em universal spinor bundle} is then defined as the associated vector bundle
\ben
\pi: \Sigma M = \tilde P \times_{\Spin_n} \Sigma_n \rightarrow \odot_+^2T^*\!M
\ee
where $\Sigma_n$ is the $n$-dimensional spin representation of $\Spin_n$. Composing $\pi$  with the fiber bundle projection $\odot_+^2T^*\!M \rightarrow M$ one can view $\Sigma M$ also as a fiber bundle over~$M$ with fiber $(\widetilde\GL_n^+ \times\Sigma_n)/\Spin_n$. Here $\Spin_n$ acts from the right on $\widetilde\GL_n^+$ through the inclusion and from the left on $\Sigma_n$ through the spin representation. Note that $\Sigma M$ is a vector bundle over $\odot_+^2T^*M$, but not over~$M$.

\medskip

A section $\Phi\in\Gamma(\Sigma M)$ determines a Riemannian metric $g=g_\Ph$ and a $g$-spinor $\phi=\phi_\Phi\in\Gamma(\Sigma_gM)$ and vice versa. 
Therefore we henceforth identify $\Phi$ with $(g,\phi)$. Next, we denote by $S(\Sigma M)$ the universal bundle of unit spinors, i.e.
\ben
S(\Sigma M)=\{\Phi \in \Sigma M : |\Phi|= 1 \}
\ee
where $|\Phi| := |\phi_\Phi|_{g_\Phi}$. Finally we introduce the spaces of smooth sections
\ben
\mc{F} = \Gamma(\Sigma M)\quad\mbox{and}\quad\mc{N} = \Gamma(S(\Sigma M))
\ee 
which we also regard as Fr\'echet fiber bundles over $\mc M$. Here and elsewhere in the article $\Gamma(\Sigma M)$ and $\Gamma(S(\Sigma M))$ denote the spaces of sections of the corresponding bundles over~$M$, and not over $\odot^2_+T^*\!M$.
%
\subsection{The Bourguignon-Gauduchon horizontal distribution}\label{subsec.bg}
%
In order to compare different fibers $\Sigma_{g_0}M$ and $\Sigma_{g_1}M$ of the universal spinor bundle over some point $x\in M$, we shall need the {\em natural horizontal distribution} of Bourguignon and Gauduchon. We refer to~\cite{bg91} for details. Let $V$ be an oriented real vector space. We denote by $\mc B V$ the set of oriented bases of $V$ which we think of as the set of orientation-preserving linear isomorphisms $\R^n\to V$. By the polar decomposition theorem $\mc BV$ is diffeomorphic to the product of the set of special orthogonal matrices and the cone of positive-definite symmetric matrices. Hence the natural fibration $p:\mc B V\to \odot^2_+V$ where the fiber $\mc B_gV$ over $g$ consists precisely of the set of oriented $g$-orthonormal bases, is trivial. Using $b\in\mc B V$ to identify $V$ with $\R^n$, the tangent space of $\mc B V$ at $b$ can be decomposed into the space tangent to the fibre $T^v_b\cong\Lambda^2\R^{n*}$, and the space of symmetric matrices $T^h_b=\odot^2\R^{n*}$ which is horizontal. Since the distribution $b\mapsto T^v_b$ is $\SO_n$-equivariant we get a connection dubbed {\em natural} by Bourguignon and Gauduchon. This construction can be immediately generalised to the universal cover $\tilde{\mc B}V\to\odot^2_+V$, where the fiber $\tilde{\mc B}_gV$ is diffeomorphic to $\Spin_n$.

\medskip

By naturality of the pointwise construction above we obtain a horizontal distribution $\mc H$ of the principal $\Spin_n$-bundle $\tilde P\to\odot^2_+M$ and consequently one of $\Sigma M$: The fiber over $g_x\in\odot^2_+M$ is given by $\mc B T_xM$. For every $\tilde b\in\tilde P_{g_x}$ lying over the basis $b$ of $T_xM$, we have a complement to the tangent space of the fiber (isomorphic to the skew-adjoint endomorphisms of $T_xM$) given by the symmetric endomorphisms of $T_xM$. Using the projection we can identify this space with $\odot^2 T_x^*M$, the tangent space of $\odot^2_+M$ at $x$.  
%
\subsection{Parallel transport}\label{par_transp}
%
The horizontal distribution just defined gives rise to a ``partial'' connection on $\Sigma M\to\odot^2_+M$ in the sense that $\nabla_X\psi\in\mc F$ is defined for a smooth spinor $\psi$ if $X$ is vertical for the bundle map $\odot^2_+M\to M$. This allows us to compare the fibers of $\Sigma_{g_0}M$ and $\Sigma_{g_1}M$ over a given point $x\in M$ along a path from $g_1(x)$ to $g_2(x)$, but not over two different points. There are two ways of making this comparison.

\medskip

Following again~\cite{bg91}, we consider for two given metrics $g_0$ and $g_1$ over $T_xM$ 
the endomorphism $A^{g_1}_{g_0}$ defined by $g_1(v,w)=g_0(A^{g_1}_{g_0}v,w)$. The endomorphism $A^{g_1}_{g_0}$ is self-adjoint and positive definite with respect to $g_0$. Consequently, it has a positive self-adjoint square root $B^{g_1}_{g_0}$. One easily checks that $A^{g_0}_{g_1}$ is the inverse of $A^{g_1}_{g_0}$, whence $(B^{g_1}_{g_0})^{-1}=B^{g_0}_{g_1}$. We map a $g_0$-orthonormal basis $b$ to $B^{g_0}_{g_1}(b)=(A^{g_0}_{g_1})^{1/2}b\in P_{g_1}$. This map can be lifted to a map $\tilde B^{g_0}_{g_1}:\tilde P_{g_0}\to\tilde P_{g_1}$ by sending the spinor $[\tilde b,\tilde\phi]$ 
to $[\tilde B^{g_0}_{g_1}\tilde b,\tilde\phi]$ which induces an isometry $\hat B^{g_0}_{g_1}:\mc F_{g_0}\to\mc F_{g_1}$. In particular, we obtain a map $\hat B^{g_0}_{g_1}: \mc N_{g_0}\to\mc N_{g_1}$. This isometry coincides with the parallel transport $\tilde P_{g_0}\to\tilde P_{g_1}$ along the path $g_t\in\odot^2_+T^*_xM$ associated with the Bourguignon--Gauduchon distribution if $g_t=(1-t)g_0+tg_1$ is just linear interpolation~\cite[Proposition 2]{bg91}.

\medskip

An alternative description of the resulting parallel transport can be also given in terms of the generalised cylinder construction from~\cite{bgm05}. This works for an arbitrary piecewise smooth path $g_t:I=[0,1]\to\mc M$ from $g_0$ to $g_1$. Let $C=I\times M$ be the Riemannian product with metric $G=dt^2+g_t$.
The spin structure on $M$ and the unique spin structure on $[0,1]$ induce a product spin structure on the cylinder. Thus we obtain a spinor bundle $\Sigma_{g_t} M$ on $M$ for $t\in[0,1]$ and a spinor bundle $\Sigma_G C$ on the cylinder. Each such spinor bundle carries a connection denoted by $\nabla^{g_t}$ resp.\ $\nabla^{C}$ coming from the Levi-Civita connection, and a Clifford multiplication
$TM\otimes \Si_{g_t} M\to \Si_{g_t} M$, $X\otimes \phi\mapsto X\cdot_t \phi$, resp.\ $TM\otimes \Si_G C\to  \Si_G C$, 
$X\otimes \phi\mapsto X*\phi$. To lighten notation we simply write $X\cdot \phi$, though Clifford multiplication actually depends on $t$. Note that our notation differs from the one in~\cite{bgm05}:
\medskip
\begin{center}
\begin{tabular}{r||c|c}
Notation in \cite{bgm05} & $X\cdot \phi$  &  $X\bullet_t \phi$    \\\hline
Present article & $X*\phi$	     &  $X\cdot \phi$
\end{tabular}
\end{center}
\medskip
We also write $M_t$ for the Riemannian manifold $(\{t\}\times M,g_t)$ and $\nu=\partial_t$ for the canonical  
vector field on $C$ which is normal to $M_t$. If $n$ is even, then $\Sigma M_t\cong\Sigma C|_{M_t}$, but Clifford multiplication is not preserved by restriction. Indeed, we have
\beq\label{clifford_mult}
X\cdot \phi=\nu * X *\phi
\ee
for $X\in\Gamma(TM)$ and $\phi\in\Gamma(\Sigma M_t)$. The same holds for $n$ odd if we set $\Sigma M=\Sigma^+C$. Parallel transport on $C$ gives rise to linear isometries $\tilde B_{g_t}^g:\Sigma_{g,x}M\to\Sigma_{g_t,x}M_t$ along the curves $t\mapsto(t,x)$ which 
coincide indeed with the maps $\tilde B_{g_t}^g$ defined above if we use linear interpolation $g_t=tg_1+(1-t)g_0$ (cf.\ Section 5 in~\cite{bgm05}). In the sequel we will therefore think of the distribution $\mc H$ at $\Phi=(g,\phi)$ as
\ben
\mc H_\Phi:=\{\left.\tfrac{d}{dt}\right|_{t=0}\hat B_{g_t}^g\phi\,:\,g_t\mbox{ a path with }g_0=g=\pi(\phi)\},
\ee 
which we can identify with the space of symmetric $2$-forms $\odot^2 T^*\!M$. Consequently we obtain the decomposition
\beq\label{splitting}
T_{(g,\phi)}\Sigma M\cong\odot^2 T^*_xM\oplus\Sigma_{g,x}M,
\ee
where $(g,\phi) \in \Sigma M$ has basepoint $x \in M$. Passing to the Fr\'echet bundle $\mc F\to\mc M$, we obtain in view of the generalised cylinder construction a horizontal distribution also denoted by $\mc H$. At $\Phi=(g,\phi)$ it is given by
\ben
\mc H_\Phi:=\{\left.\tfrac{d}{dt}\right|_{t=0}\tilde B_{g_t}^g\phi\,:\,g_t\mbox{ a path with }g_0=g=\pi(\phi)\}.
\ee 
It follows that
\ben
T_\Phi\mc F=\mc H_\Phi\oplus T_\Phi\mc F_g\cong\Gamma(\odot^2T^*\!M)\oplus \Gamma(\Sigma_gM),
\ee
and the subspace $T_\Phi\mc N$ corresponds to 
\ben 
\Gamma(\odot^2T^*M)\oplus\Gamma(\phi_\Phi^\perp)=\{(h,\psi)\in\Gamma(\odot^2T^*\!M)\oplus \Gamma(\Sigma_gM)\,:\, \<\psi(x),\phi_\Phi(x)\>=0\,\forall x\in M\}.
\ee
%
%
\section{The spinorial energy functional}\label{spindir}
%
%
In this section we introduce the spinorial energy functional and investigate its basic properties. 

\begin{definition}
The {\em energy functional} $\mc{E}$ is defined by
\ben
\mc{E}:\mc{N}\to\R_{\geq 0},\quad \Phi \mapsto\tfrac{1}{2}\int_M|\nabla^g\phi|_g^2\,dv^g,
\ee
where, as above, we identify an element $\Phi \in \mc{N}$ with the pair $(g,\phi)=(g_\Phi,\phi_\Phi)$ for $g \in \Gamma(\odot^2_+T^*\!M)$ and $\phi \in \Gamma(\Sigma_gM)$. 
\end{definition}
%
\subsection{Symmetries of the functional}\label{symmetries}
%
\subsubsection{Rescaling}
Consider the action of $c \in \R_+$ on the cone of metrics by rescaling $g\mapsto c^2 g$. This conformal change of the metric can be canonically lifted to the spinor bundle. In the notation of Subsection~\ref{subsec.bg} we have $A^{c^2g}_g= c^2 \mr{Id}$, $B^{c^2 g}_g= c \mr{Id}$, $B^g_{c^2 g}= c^{-1} \mr{Id}$. As before we obtain maps  $\tilde B^g_{c^2 g}:\tilde P_g\to\tilde P_{c^2 g}$ and a fiberwise isometry $\hat B^g_{c^2 g}:\Sigma_g\to\Sigma_{c^2g}$ sending $[\tilde b,\tilde\phi]$ to $[\tilde B^g_{c^2 g}(\tilde b),\tilde\phi]$. In particular, we get the bundle map
\ben
\hat B^g_{c^2 g}:\mc N\to\mc N,\quad\phi\in\mc N_g\mapsto\hat B^g_{c^2 g}\phi\in\mc N_{c^2g}.
\ee

\begin{prop}\label{compat1}
Let $\Phi=(g,\phi) \in \mc{N}$. Then $\mc{E}(c^2g,\hat B^g_{c^2 g}\phi)=c^{n-2}\mc{E}(g,\phi)$ for all $c>0$.
\end{prop}
\begin{proof}
If $\tilde g=c^2 g$ for some constant $c>0$, then $\nabla^{\tilde g}=\nabla^g$ 
(cf.\ for instance Lemma II.5.27 in~\cite{lami89}). Further, $dv^{\tilde g}=c^n\, dv^g$ and $\ti e_k=c^{-1} e_k$, so that 
\ben
\mc{E}(c^2g,\hat B^g_{c^2 g}\phi)=\tfrac{1}{2}\int_M\sum_{k=1}^n|\nabla^{\tilde g}_{\tilde e_k}\hat B^g_{c^2 g}\phi|^2\,dv^{\tilde{g}}
=\tfrac{1}{2}\int_M\sum_{k=1}^n|\nabla^{g}_{c^{-1} e_k}\phi|^2 c^n \,dv^g=c^{n-2}\mc{E}(g,\phi)
\ee
which is what we wanted to show.
\end{proof}
\subsubsection{Action of the spin-diffeomorphism group}
An orientation preserving diffeomorphism $f: M \rightarrow M$ induces a bundle map ${df}:P\to P$, where ${df}$ maps an oriented frame $(v_1,\ldots,v_n)$ over~$x$ to the oriented frame $(df(v_1),\ldots,df(v_n))$ over $f(x)$. Then $f$ is called {\em spin structure preserving} if it lifts to a bundle map $F: \tilde P \rightarrow \tilde P$ of the spin structure $\tilde P\to P$ making the diagram
\beq\label{diagram}
\begin{CD}
\tilde P @>F>> \tilde P\\
@VVV @VVV\\
P @>df>> P
\end{CD}
\bigskip
\ee
commutative. In other words, an orientation preserving diffeomorphism $f: M \rightarrow M$ is spin structure preserving if and only if the pullback of $\ti P\to P$ under $f$ is a spin structure on $M$ that is equivalent to $\ti P\to P$. We denote by $\diff_s(M)\subset\diff_+(M)$ the group of spin structure preserving diffeomorphisms. A {\em spin-diffeomorphism} is a diffeomorphism $F: \tilde P \rightarrow \tilde P$ making the diagram~\eqref{diagram} commutative for some orientation preserving diffeomorphism $f: M \rightarrow M$. Since $\tilde P \rightarrow P$ is a $\mZ_2$-principal bundle the lift is determined up to a $\Z_2$-action. Put differently the group of all spin-diffeomorphisms which we denote by $\widehat\diff_s(M)$, is an extension of $\diff_s(M)$ by $\mZ_2=\{ \pm 1\}$, i.e.\ it fits into a short exact sequence
\ben
1\rightarrow \{ \pm 1\} \rightarrow \widehat\diff_s(M)\rightarrow \diff_s(M) \rightarrow 1.
\ee
A special situation arises if $f\in\diff_0(M)$, i.e.\ $f$ is isotopic to the identity. Since the homotopy can be lifted to $\tilde P$ and homotopic isotopies yield the same lift, we obtain a map $\widetilde\diff_0(M)\to \widehat\diff_s(M)$ from the universal covering group of  $\diff_0(M)$. 

\medskip

If $g$ is a Riemannian metric on $M$, then $F \in \widehat\diff_s(M)$ covering $f \in \diff_s(M)$ maps $\tilde P_g$ to $\tilde P_{f_*g}$, where by definition $f_*g:=(f^{-1})^*g$. For any $\phi\in\mc F_g$ we obtain a spinor 
field $F_*\phi \in \mc{F}_{f_*g}$ as follows: If $\phi$ is locally expressed as  $[\tilde b,\tilde\phi]$, then $F_*\phi$ is locally expressed as $[ F\circ \tilde b \circ f^{-1}, \tilde \phi\circ f^{-1}]$. Since
\beq\label{thetaisom}
|F_*\phi|_{f_*g}(x)=\<\tilde\phi\circ f^{-1}(x),\tilde\phi\circ f^{-1}(x)\>=|\phi|_g(f^{-1}(x)),
\ee
$F_*$ is a map from $\mc N_g$ to $\mc N_{f_*g}$ with inverse $(F^{-1})_*$. Consequently, we obtain the bundle map
\ben
F_*:\mc N\to\mc N,\quad\phi\in\mc N_g\mapsto F_*\phi\in\mc N_{f_*g}.
\ee
Note that $(F_1 \circ F_2)_* = F_{1*} \circ F_{2*}$. Since the spinorial energy functional only depends on the metric and the spinor bundle which both transform naturally under spin-diffeomorphisms, we immediately conclude the following

\begin{prop}\label{compat2}
Let $\Phi=(g,\phi) \in \mc{N}$. Then  $\mc{E}(F_* \Phi) = \mc{E}(\Phi)$ for all $F\in\widehat\diff_s(M)$.
\end{prop}

Finally we discuss the infinitesimal action of $\widehat\diff_s(M)$ on $\mc F$. Consider a vector field $X \in \Gamma(TM)$ with associated flow $f_t \in \diff_0(M)$, that is $\frac{d}{dt}f_t = X \circ f_t$ with $f_0 = \Id_M$. Hence $f_t$ lifts to a $1$-parameter family $F_t \in \widehat \diff_s(M)$ with $F_0 = \Id_{\tilde P}$. If $F\in\widehat\diff_s(M)$ we define $F^*:\mc N\to\mc N$ by $F^*=F^{-1}_*$, so that $F^*\phi\in\mc N_{f^*g}$ if $F$ covers $f\in\diff(M)$. Hence $F_t^*\Phi$ is a family of sections of $\Sigma M$ which we can differentiate with respect to $t$. Using the connection on the bundle $\mc F$ we may split the resulting ``Lie derivative"
\ben
\left.\tfrac{d}{dt}\right|_{t=0} F_t^*\Phi \in T_{\Phi} \mc F = \mc H_\Phi \oplus T_\Phi\mc F_g
\ee
into the horizontal part
\ben
\mc L_X g = \left.\tfrac{d}{dt}\right|_{t=0} f_t^*g \in\Gamma(\odot^2T^*\!M) = \mc H_\Phi
\ee
and the vertical part
\ben
\tilde{\mc L}^g_X \phi := \left.\tfrac{d}{dt}\right|_{t=0} \hat B_g^{f_t^*g}F_t^*\Phi \in \Gamma(\Sigma_gM) = T_\Phi\mc F_g.
\ee
This vertical part $\tilde{\mc L}^g_X \phi$ is called the {\em metric Lie derivative}. It is the lift from the metric Lie derivative $\mc L^g$ on tensor fields which is defined using the metric, and satisfies, in particular, $\mc L^g_Xg=0$ (cf.~\cite{bg91}). By~\cite[Proposition 17]{bg91} we have
\beq\label{spinlie}
\tilde{\mc L}^g_X\phi=\nabla^g_X\phi-\tfrac{1}{4}dX^\flat\cdot\phi.
\ee
More generally, consider a curve $\Phi_t=(g_t,\phi_t)$ in $\mc F$ with $(\dot g_t,\dot\phi_t):=\frac{d}{dt}\Phi_t\in T_{\Phi_t}\mc F$ and $X_t$ a time-dependent vector field whose flow lifts to $F_t \in \widehat \diff_s(M)$. Since $F^*$ is linear on the fibres as a map $\mc F\to\mc F$, we get
\beq\label{pathder}
\tfrac{d}{dt}F_t^*\Phi_t=F_t^*(\mc L_{X_t}g_t+\dot g_t,\tilde{\mc L}^g_{X_t}\phi_t+\dot\phi_t).
\ee
\subsubsection{Symmetries from pointwise representation theory}\label{pointrepsym}
Let $P_G\to M$ be a principal $G$-fibre bundle with connection, $V$ a $G$-representation space and $L\in\End(V)$ a $G$-equivariant map. Then it follows from general principal fibre bundle theory that $L$ gives rise to a well-defined bundle endomorphism of $P_G\times_GV$ which is parallel with respect to the induced covariant derivative $\nabla$. Using the representation theory of spinors we will compute the $G=\Spin(n)$-equivariant maps for the spin representations. Those which in addition preserve the inner products will therefore preserve the energy functional, i.e.\ $\mc E(g,L(\phi))=\mc E(g,\phi)$, cf. Tables~\ref{isosigmareal} and~\ref{isosigmacomp} below. For example, the complex volume form $\vol^\C:=i^{n(n+1)/2} \vol_g$ acts as an isometry on $\Sigma_n$ via Clifford multiplication and commutes with $\Spin_n$. Further, if $n$ is even, it defines an involution so that not only $\vol^\C$ is preserved under $\nabla^g$, but also the decomposition $\Sigma_n=\Sigma^+_n\oplus\Sigma^-_n$ into the $\pm$-eigenspaces of positive and negative eigenspinors (these are actually the irreducible $\Spin_n$-representations mentioned in Section~\ref{spingeo}). Consequently, a positive or negative spinor is a critical point of $\mc E$ if and only if it is a critical point of the restriction of $\mc E$ to positive and negative spinors. A further application will be discussed in Section~\ref{Dirichlet}. To lighten notation we shall drop any reference to background metrics.

\medskip

Since we work with a real inner product on spinors it will be convenient to work with $\R$-linear maps. Though we mainly consider complex spin representations we start by considering the {\em real\/} spin representations $\Sigma_n^\R$. By definition, these are obtained by restricting the irreducible real representations of $\cliff_n$ to $\Spin_n$ (cf.~\cite[Definition I.5.11]{lami89}). The algebra $\End(\Sigma^\R_n)^{\Spin_n}$ of $\Spin_n$-equivariant endomorphisms of $\Sigma^\R_n$ can be computed from the decomposition of $\Sigma_n^\R$ into irreducible components and their types. Recall that by Schur's lemma for an irreducible real representation $V$ of $\Spin_n$ the algebra $\End(V)^{\Spin_n}$ is either $\R$, $\C$ or $\Qa$ which by definition is the {\em type} of the representation. Note that $\Sigma_n^\R$ is not always irreducible as a representation of $\Spin_n$, see~\cite[Proposition I.5.12]{lami89} for the precise decomposition. However, the irreducible components always have the same type, which we will refer to as the type of $\Sigma_n^\R$.
As $\Spin_n$ generates $\cliff_n^{ev}$ as an algebra, an endomorphism
is $\cliff_n^{ev}$ equivariant if and only if it is $\Spin_n$ equivariant. 
Thus the type of $\Sigma_n^\R$ is determined by the representation of the even part of $\cliff_n$ which as an algebra is just $\cliff_{n-1}$. Hence the type of $\Sigma_n^\R$ can be read off from Table III in~\cite[Section I.5]{lami89} (cf.\ also~\cite[Remark I.5.13]{lami89}). With this data at hand, $\End(\Sigma^\R_n)^{\Spin_n}$ and the group $\mr{Isom}(\Sigma_n^\R)^{\Spin_n}$ of $\Spin_n$-equivariant isometries of $\Sigma_n^\R$ can easily be determined, see Table~\ref{isosigmareal}. 
\begin{table}[htb]
\begin{center}
\begin{tabular}{c|c|c|c|c}
$n \mod 8$ & type of $\Sigma_n^\R$ & decomposition& $\End(\Sigma_n^\R)^{\Spin_n}$ & $\mr{Isom}(\Sigma_n^\R)^{\Spin_n}$\\
\hline
0 & $\R$ & $\Delta_+\oplus\Delta_-$ & $\R(1)\oplus\R(1)$ & $\mr{O}(1)\times\mr{O}(1)$\\
1 & $\R$ & $\Delta\oplus\Delta$ & $\R(2)$ & $\mr{O}(2)$\\
2 & $\C$ & $\Delta\oplus\Delta$ & $\C(2)$ &  $\mr{U}(2)$\\
3 & $\Qa$ & $\Delta$ & $\Qa(1)$ & $\Sp(1)$\\
4 & $\Qa$ & $\Delta_+\oplus\Delta_-$ & $\Qa(1)\oplus\Qa(1)$ & $\Sp(1)\times\Sp(1)$\\
5 & $\Qa$ & $\Delta$ & $\Qa(1)$ & $\Sp(1)$\\
6 & $\C$ & $\Delta$ & $\C(1)$ & $\mr{U}(1)$\\
7 & $\R$ & $\Delta$ & $\R(1)$ & $\mr{O}(1)$
\end{tabular}
\end{center}
\vspace{7pt}
\caption{Isomorphism classes of $\mr{Isom}(\Sigma_n^\R)^{\Spin_n}$}\label{isosigmareal}
\end{table}
Here, $\mathbb{K}(i)$ denotes the $i\times i$-matrices with coefficients in $\mathbb K=\R$, $\C$ or $\Qa$. Further, $\Delta\oplus\Delta$ means that $\Sigma_n^\R$ is the sum of two equivalent irreducible $\Spin_n$-representations, while $\Delta_+\oplus\Delta_-$ is an irreducible decomposition into non-equivalent ones. For the remaining cases $\Sigma^\R_n$ is irreducible. The algebra $\End_\R(r\Sigma_n)^{\Spin_n}$ of $\R$-linear $\Spin_n$-equivariant endomorphisms of $\Sigma_n$  and the group of $\R$-linear $\Spin_n$-equivariant isometries $\mathrm{Isom}_\R(\Sigma_n)^{\Spin_n}$ can now be computed as follows. The complex representation $\Sigma_n$ is obtained as the complexification of $\Sigma_n^\R$ if $n \equiv 0,6$ or $7\mod 8$ so that as a real representation $\Sigma_n=\Sigma_n^\R \oplus \Sigma_n^\R$. In the remaining cases $\Sigma_n=\Sigma_n^\R$, and therefore $\End_\R(\Sigma_n)^{\Spin_n} =\End(\Sigma_n^\R)^{\Spin_n}$. Now if for instance $n\equiv6\mod8$, then $\Sigma_n^\R=\Delta$ is irreducible and $\End(\Sigma_n^\R)^{\Spin_n}=\C(1)$ according to Table~\ref{isosigmareal}. Hence $\Sigma_n=\Delta \oplus \Delta$ as a real representation and $\End_\R(\Sigma_n)^{\Spin_n}=\C(2)$ as a real algebra. Continuing in this vein we arrive at Table~\ref{isosigmacomp}.

\begin{table}
\begin{center}
\begin{tabular}{c|c|c}
$n\mod 8$ & $\End_\R(\Sigma_n)^{\Spin_n}$ & $\mr{Isom}_\R(\Sigma_n)^{\Spin_n}$ \\
\hline
0 & $\R(2)\times\R(2)$ & $\mr{O}(2)\times \mr{O}(2)$\\
1 & $\R(2)$ & $\mr{O}(2)$\\
2 & $\C(2)$ & $\mr{U}(2)$\\
3 & $\Qa(1)$ & $\Sp(1)$ \\
4 & $\Qa(1)\oplus\Qa(1)$ & $ \Sp(1)\times \Sp(1)$\\
5 & $\Qa(1)$&  $\Sp(1)$\\
6 & $\C(2)$ & $\mr{U}(2)$\\
7 & $\R(2)$ &  $\mr{O}(2)$
\end{tabular}
\end{center}
\vspace{7pt}
\caption{Isomorphism classes of $\mr{Isom}_\R(\Sigma_n)^{\Spin_n}$}\label{isosigmacomp}
\end{table}

\medskip

For later applications we consider some concrete elements in $\mr{Isom}_\R(\Sigma_n)^{\Spin_n}$. Obvious ones are scalar multiplication by $-\Id\in\mr{Isom}_\R(\Sigma_n)^{\Spin_n}$ and $S^1\subset\mr{Isom}_\R(\Sigma_n)^{\Spin_n}$ as well as the action by the volume element $\vol_g$. Further elements of $\mr{Isom}_\R(\Sigma_n)^{\Spin_n}$ are provided by {\em quaternionic} and {\em real} structures, i.e.\ complex anti-linear maps $J$ with $J^2=\Id$ and $J^2=-\Id$ respectively. For $n\equiv0$, $1$, $6$ or $7\mod 8$ there exists a $\Spin_n$-equivariant quaternionic structure $J_n$ on $\Sigma_n$ while there exists a $\Spin_n$-equivariant real structure $J_n$ for $n\equiv2$, $3$, $4$ or $5\mod 8$~\cite[Sec.~1.7]{fr00}. For $n$ is even we obtain further real and quaternionic
structures $\tilde J_n:=J_n(\vol\cdot.)\in\mr{Isom}_\R(\Sigma_n)^{\Spin_n}$, namely $\tilde J_n^2=1$ for $n\equiv 0,\,2\mod 8$ and $\tilde J_n^2=-1$ for $n\equiv 4,\,6\mod 8$~\cite[Chapter 2]{he12}.

\begin{cor}\label{reality}
On $\Sigma_n$ there are real structures in dimensions $n\equiv 0,1,2,6,7\mod 8$ and quaternionic structures in dimensions $n\equiv 2,3,4,5,6\mod 8$ preserving $\mc E$.
\end{cor}

Together with scalar multiplication by $S^1$ each real structure yields a subgroup of $\mr{Isom}_\R(\Sigma_n)^{\Spin_n}$ isomorphic to $S^1\rtimes\mZ_2\cong\mr{O}(2)$, where $\Z_2$ acts by conjugation on $S^1$. On the other hand, a quaternionic structure turns $\Sigma_n$ into a quaternionic vector space so we get an induced equivariant and isometric action of $\Sp(1)$. This actually gives the whole group $\mr{Isom}_\R(\Sigma_n)^{\Spin_n}$ for $n$ odd. If $n\equiv 2\mod 8$, then the action of $J_n$ and $i$ generates a subgroup isomorphic to $\SU(2)$. Further, $\vol_g$ induces a complex linear map with $\vol^2_g=-\Id$ and which commutes with~$J_n$ and~$i$. One easily checks that $T_\alpha:=(\cos \alpha)\Id + (\sin \alpha)\vol\in\mr{Isom}(r\Sigma_n)^{\Spin_n}$. As $T_\pi$ coincides with $-\Id\in \SU(2)$ we obtain the group $S^1\times_{\mZ_2} \SU(2)\cong\mr{U}(2)$. For $n\equiv 6 \mod 8$ the symmetries $\tilde J_n$ and $i$ give rise to an $\SU(2)$-action on the unit spinors. The group generated by $J_n$ and this $\SU(2)$ is a semi-direct product rather than a direct product of $\Z_2$ with $\SU(2)$, for $J_n$ and $i$ anti-commute. For $n\equiv 4\mod 8$ we get a proper subgroup of $\mr{Isom}_\R(\Sigma_n)^{\Spin_n}$ isomorphic to $\mZ_2\times \SU(2)$ while for $n\equiv 0 \mod 8$ we get a proper subgroup isomorphic to $\Z_2\times\mr{O}(2)$. 
%
\subsection{Critical points}\label{critical}
%
Next we determine the critical points and compute the $L^2$-gradient of $\mc{E}$.

\medskip

Let $\Phi = (g,\phi) \in \mc N$. In view of the splitting~\eqref{splitting}, any element in $T_\Phi\mc N$ can be decomposed into a vertical part $\dot \phi = \left.\frac{d}{dt}\right|_{t=0} \phi_t$ for a curve $\phi_t$ in $\mc N_g$ with $\phi_0=\phi$, and a horizontal part $\dot g^{hor} = \left.\frac{d}{dt}\right|_{t=0} g^{hor}_t$, where $g^{hor}_t$ is the horizontal lift of a curve~$g_t$ in $\mc M$ with $g_0=g$ starting at $\Phi$. In the following we identify $\dot g^{hor} \in T_\Phi \mc N$ with $\dot g = \left.\frac{d}{dt}\right|_{t=0} g_t \in T_g \mc M = \Gamma(\odot^2T^*\!M)$ without further notice. The pointwise constraint $|\phi_t|=1$ implies that $\dot \phi$ must be pointwise perpendicular to $\phi$, i.e.\ $\dot\phi\in \Gamma(\phi^{\perp})$. We define $Q_1(\Phi)\in\Gamma(\odot^2T^*\!M)$ via
\ben
\llangle Q_1(\Phi),\dot g \rrangle_g:=-\left.\tfrac{d}{dt}\right|_{t=0}\mc{E}(g^{hor}_t)=-(D_\Phi\mc{E})(\dot g^{hor},0).
\ee
Up to the minus sign this is just the {\em energy-momentum tensor} for $\phi$ as defined in~\cite{bgm05}. Similarly, we introduce
\ben
\llangle Q_2(\Phi),\dot\phi \rrangle_g:=-\left.\tfrac{d}{dt}\right|_{t=0}\mc{E}(g, \phi_t) = -(D_\Phi\mc{E})(0,\dot \phi)
\ee
where $ \llangle \phi,\psi \rrangle_g=\int_M \langle\phi,\psi \rangle \,dv^g$ denotes the natural $L^2$-inner product on $\mc F_g$. Then under the identification of~\eqref{splitting},
\ben
Q:\mc N\to T\mc N,\quad\Phi\in\mc N \mapsto\bigl(Q_1(\Phi),Q_2(\Phi)\bigr)\in\Gamma(\odot^2T^*\!M) \times \Gamma(\phi^{\perp})
\ee
becomes the negative gradient of $\mc{E}$. 

\medskip

The compatibility of $\mc{E}$ with the actions of $\R_+$ and $\widehat\diff_s(M)$ implies the following equivariance properties of $Q$. In particular, we obtain

\begin{cor}
Let $\Phi=(g,\phi) \in \mc{N}$. Then

\smallskip

\noindent (i) $Q_1(c^2g,\hat B^g_{c^2g}\phi) = Q_1(g,\phi)$ and $Q_2(c^2g,\hat B^g_{c^2g}\phi) = c^{-2}Q_2(g,\phi)$ for all $c>0$.

\smallskip

\noindent (ii) $Q(F_*\Phi) = F_*Q(\Phi)$ for all $F \in \widehat\diff_s(M)$.

\end{cor}

\begin{proof}
(i) By definition of the negative gradient one has
\begin{align*}
\llangle Q_1(c^2g,\hat B^g_{c^2g}\phi), c^2 \dot g \rrangle_{c^2g} &= - (D_{(c^2g,B^g_{c^2g}\phi)} \mc{E}) (c^2\dot g,0)\\
&= - \left.\tfrac{d}{dt}\right\vert_{t=0} \mc{E} (c^2(g+t\dot g),B^g_{c^2(g+t\dot g)}\phi)\\
&= - c^{n-2} (D_{(g,\phi)}\mc{E})(\dot g,0) = c^{n-2} \llangle Q_1(g,\phi),\dot g\rrangle_g,
\end{align*}
where for the last line we have used Proposition~\ref{compat1}. On the other hand,
\ben
\llangle Q_1(c^2g,\hat B^g_{c^2g}\phi), c^2 \dot g \rrangle_{c^2g} = c^{n-2} \llangle Q_1(c^2g,\hat B^g_{c^2g}\phi),\dot g \rrangle_g,
\ee
whence the result for $Q_1$. Similarly,
\ben
c^{n} \llangle Q_2(c^2g,\hat B^g_{c^2g}\phi), \dot \phi \rrangle_{g} =  \llangle Q_2(c^2g,\hat B^g_{c^2g}\phi), \dot \phi \rrangle_{c^2g}= c^{n-2} \llangle Q_2(g,\phi),\dot \phi \rrangle_g.
\ee

\noindent (ii) Follows from Proposition \ref{compat2}.
\end{proof}

The formal adjoint $\delta_g^*$ of the divergence $\delta_g: \Gamma(\odot^2T^*\!M) \to \Gamma(T^*\!M)$, is the symmetrisation of the covariant derivative, i.e.\ $(\delta_g^* \eta)(X,Y) = \frac{1}{2}((\nabla^g_X \eta)(Y) + (\nabla^g_Y \eta)(X))$ for $\eta \in \Gamma(T^*\!M)$. In particular, $\mathcal L_X g = 2 \delta_g^* X^\flat$. Since $\delta^*_g$ is overdetermined elliptic we have an orthogonal decomposition $\Gamma(\odot^2T^*\!M) = \ker \delta_g \oplus \im \delta_g^*$ (see e.g.\ Appendix I, Corollary 32 in \cite{be87}). If we identify $T_g\mc M = \Gamma(\odot^2T^*\!M)$, then $\im \delta_g^*$ is tangent to the $\diff(M)$-orbit through $g$ and $\ker \delta_g$ provides a natural complement in $T_g\mc M$. In a similar fashion we consider the operator
\begin{align*}
\lambda^*_{g,\phi}: \Gamma(TM) &\rightarrow T_{(g,\phi)}\mc F=\Gamma(\odot^2T^*\!M) \oplus \Gamma(\Sigma_gM)\\
X &\mapsto (\mc L_X g, \tilde {\mc L}^g_X \phi) = (2 \delta_g^* X^\flat, \nabla^g_X\phi-\tfrac{1}{4} dX^\flat \cdot \phi).
\end{align*}

Consider the fiberwise linear maps $A_\varphi: TM \to \Sigma_gM, X \mapsto \nabla_X \varphi$ and $B_\varphi: \Lambda^2T^*\!M \to \Sigma_gM, \omega \mapsto \omega \cdot \varphi$. Expressed in a local orthonormal frame $\{e_i\}$, their pointwise adjoints $A_\varphi^*$ and $B_\varphi^*$ are given by
\ben
A_\phi^*(\dot \phi) = \sum_i \langle \dot \phi, \nabla_{e_i}^g \phi \rangle e_i \quad\text{and}\quad
B_\phi^*(\dot \phi) = \sum_{i < \j} \langle \dot \phi, e_i \wedge e_j \cdot \phi \rangle e_i \wedge e_j.
\ee
Hence $\lambda_{g,\phi}^*(X) = (2 \delta_g^* X^\flat, A_\phi (X) - \frac{1}{4} B_\phi (dX^\flat))$ and further
\begin{align*}
\lambda_{g,\phi}: \Gamma(\odot^2T^*\!M) \oplus \Gamma(\Sigma_gM) &\rightarrow \Gamma(TM)\\
(\dot g, \dot \phi) &\mapsto 2 (\delta_g \dot g)^\sharp + A_\phi^* (\dot \phi) - \tfrac{1}{4} (\delta_g B_\phi^* (\dot \phi))^\sharp
\end{align*}
for the formal adjoint of $\lambda^*_{g,\phi}$. As in the case of metrics, $\im \lambda^*_{g,\phi}$ is tangent to the $\widehat\diff_s(M)$-orbit through $(g,\phi)$. Note that $\im \lambda^*_{g,\phi}$ is actually contained in $T_{(g,\phi)} \mc N = \Gamma(\odot^2T^*\!M) \oplus \Gamma(\phi^{\perp})$, since the $\widehat\diff_s(M)$-action preserves $\mc N \subset \mc F$.

\begin{lemma}\label{orth_decomp}
There are orthogonal decompositions:

\smallskip

\noindent (i) $T_{(g,\phi)} \mc F = \ker \lambda_{g,\phi} \oplus \im \lambda_{g,\phi}^*$ for $(g,\phi) \in \mc F$.

\smallskip

\noindent (ii) $T_{(g,\phi)} \mc N = \ker \lambda_{g,\phi} \cap T_{(g,\phi)}\mc N \oplus \im \lambda_{g,\phi}^*$ for $(g,\phi) \in \mc N$.

\end{lemma}

\begin{proof}
(i) The symbol of $\lambda^*_{g,\phi}$ at $\xi \in T^*\!M_x$ is given by
\ben
\sigma_\xi (\lambda^*_{g,\phi})v = i ( \xi \otimes v^\flat + v^\flat \otimes \xi, -\tfrac{1}{4} \xi \wedge v^\flat \cdot \phi)
\ee
for $v \in T_xM$ and is clearly injective. Hence $\lambda^*_{g,\phi}$ is overdetermined elliptic and the result follows as above.

\smallskip

\noindent (ii) If $\phi\in\mc N_g$ is a unit spinor, then $\tilde{\mc{L}}_X\phi \in \Gamma(\phi^{\perp})$ for any $X \in \Gamma(TM)$. Hence the image of $\lambda_{g,\phi}^*$ is contained in $T_{(g,\phi)}\mc N$ and the result follows from (i).   
\end{proof}

Just as the classical Bianchi identity of the Riemann curvature tensor is a consequence of its $\diff(M)$-equivariance~\cite{ka81}, we note the following corollary to Proposition~\ref{compat2}.

\begin{cor}[Bianchi identity]\label{bianchi} 
Let $(g,\phi) \in \mc N$. Then $\lambda_{g,\phi} Q(g,\phi)=0$.
\end{cor}

\begin{proof}
Since $\mc E$ is invariant under $\widehat\diff_s(M)$ according to Proposition~\ref{compat2}, the orbit $\widehat\diff_s(M) \cdot  (g,\phi)$ is contained in the level set $\mc E^{-1}(\mc E(g,\phi))$. Now the (negative) gradient of $\mc E$ is orthogonal to the level set, hence in particular to $\widehat\diff_s(M) \cdot  (g,\phi)$. At $(g,\phi)$, the tangent space to this orbit is given by $\im \lambda^*_{g,\phi}$. The result follows in view of Lemma \ref{orth_decomp}.
\end{proof}

We recall the definition of the {\em connection Laplacian} $\nabla^{g*}\nabla^g:\mc F_g\to\mc F_g$, namely
\ben
\nabla^{g\ast}\nabla^g\phi=-\sum_{j=1}^n(\nabla^g_{e_j}\nabla^g_{e_j}\phi-\nabla^g_{\nabla^g_{e_j}e_j}\phi),
\ee
where $e_1,\ldots,e_n$ is a local orthonormal basis. Note that $\llangle\nabla^g\phi,\nabla^g\phi\rrangle_g=\llangle\nabla^{g\ast}\nabla^g\phi,\phi\rrangle_g$. For a tensor $T$ let
\ben
\div_g T=-\sum_{k=1}^n(\nabla^g_{e_k}T)(e_k,\cdot)
\ee
be the {\em divergence} of $T$. 

\begin{thm}\label{critpointthm}
For $\Phi = (g,\phi)\in \mc{N}$ we have
\ben
\begin{array}{l}
Q_1(\Phi)=-\tfrac{1}{4}|\nabla^g\phi|^2_gg-\tfrac{1}{4}\div_g T_{g,\phi}+\tfrac{1}{2}\langle\nabla^g\phi\otimes\nabla^g\phi\rangle,\\[5pt]
Q_2(\Phi)=-\nabla^{g\ast}\nabla^g\phi+|\nabla^g\phi|^2_g\phi,
\end{array}
\ee
where $T_{g,\phi}\in\Gamma(T^*\!M\otimes\odot^2T^*\!M)$ is the symmetrisation 
(as defined in Section~\ref{spingeo}) 
in the second and third component of the $(3,0)$-tensor defined by $\<(X\wedge Y)\cdot\phi,\nabla^g_Z\phi\>$ for $X$, $Y$ and $Z$ in $\Gamma(TM)$. Further, $\langle\nabla^g\phi\otimes\nabla^g\phi\rangle$ is the symmetric $2$-tensor defined by $\langle\nabla^g\phi\otimes\nabla^g\phi\rangle(X,Y)=\langle\nabla^g_X\phi,\nabla^g_Y\phi\rangle$ for $X,Y \in \Gamma(TM)$.
\end{thm}

\begin{rem*}
Let $D_g:\mc F_g\to\mc F_g$ be the {\em Dirac operator} associated with the spin structure, i.e.\ locally $D_g\phi=\sum_k e_k\cdot\nabla^g_{e_k}\phi$. A pointwise computation with $\nabla^ge_k(x)=0$ implies that $\tr_g\div_g T_{g,\phi}=\<D_g^2\phi,\phi\>-|D_g\phi|^2$, that is, the trace of the divergence term in $Q_1$ measures the pointwise failure of self-adjointness of $D_g$.
\end{rem*}

\begin{proof}
The vertical variation of $\mc{E}$ which gives $Q_2$ is the easy part. For $g$ a fixed Riemannian metric let $\phi_t\in\mc N_g$ be a smooth family of spinors with $\phi_0=\phi$. Set $\dot \phi = \left. \frac{d}{dt}\right\vert_{t=0} \phi_t$. Then
\ben
\left.\tfrac{d}{dt}\right\vert_{t=0} \mc{E}(g,\phi_t)=\frac{1}{2}\int_M\left.\tfrac{d}{dt}\right\vert_{t=0} |\nabla^g\phi_t|_g^2 \, dv^g=\int_M (\nabla^g\dot\phi,\nabla^g\phi)_g \, dv^g.
\ee
To get $Q_2(\Phi)$ it remains to determine the component of $\nabla^{g*}\nabla^g\phi$ orthogonal to $\phi$. Since $\<\phi,\nabla^g_X\phi\>=0$ for all $X \in \Gamma(TM)$, we have
\ben
\<\varphi,\nabla^{g\ast}\nabla^g\phi\>=-\sum_k\<\phi,\nabla^g_{e_k}\nabla^g_{e_k}\phi\>=|\nabla^g\phi|_g^2,
\ee
and we get the asserted formula for $Q_2$.

\smallskip

Secondly, we calculate $Q_1$. Let $g_t$ be a smooth family of Riemannian metrics with $g_0=g$. Set $\dot g = \left. \frac{d}{dt}\right\vert_{t=0}g_t$. Further, let $g_t^{hor}=(g_t,\phi_t)$ be the horizontal lift of the family $g_t$ to $\mc{F}$. Here and in the following we use the notation of Section \ref{par_transp}. Then
\beq\label{emtensor}
\left.\tfrac{d}{dt}\right\vert_{t=0} \mc{E}(g_t^{hor}) = \tfrac{1}{4}\int_M|\nabla^g\phi|_g^2\,\tr_g\dot g \, dv^g + \tfrac{1}{2}\int_M\left.\tfrac{d}{dt}\right\vert_{t=0} \vert\nabla^{g_t}\phi_t|_{g_t}^2 \, dv^g,
\ee
where we have used the standard variation formula $\left.\tfrac{d}{dt}\right\vert_{t=0} dv^{g_t} = \frac{1}{2}\tr_g\dot g \,dv^g$. For the second term we first proceed pointwise and fix a local orthonormal basis $e_{k,t}$ in $C$ around $x\in \{0\} \times M$ with $e_{0,t}=\nu$. For $e_{k,0}$ we simply write $e_k$. We may assume that $(\nabla^g e_k)(x) = 0$ and $\nabla^C_\nu e_{k,t} = 0$ for $k=1, \ldots,n$. It follows that  
\ben
|\nabla^{g_t}\phi_t|^2_{g_t}=\sum_{k=1}^n|\nabla^{g_t}_{e_{k,t}}\phi_t|^2
\ee
and thus
\ben
\left.\tfrac{d}{dt}\right\vert_{t=0} |\nabla^{g_t}\phi_t|^2_{g_t}=2\sum_{k=1}^n\left.\<\nabla^C_\nu\nabla^{g_t}_{e_{k,t}}\phi_t,\nabla^g_{e_k}\phi\>\,\right|_{t=0}
\ee
(recall that $\nu=\partial_t$). On the other hand let $W_t$ denote the Weingarten map $TM_t\to TM_t$ for the hypersurface $M_t=\{t\} \times M$ defined through
\ben
\nabla^{C}_XY=\nabla^{g_t}_XY+g_t(W_t(X),Y)\nu.
\ee
We identify $\{0\} \times M$ with $M$ and simply write $W$ for $W_0$ in the following.
By (3.5) in~\cite{bgm05} and equation \eqref{clifford_mult} we have
\beq\label{cov_der}
\nabla^{g_t}_X\phi_t=\nabla^C_X\phi_t+\tfrac{1}{2}\nu\ast W_t(X)\ast\phi_t =\nabla^C_X\phi_t+\tfrac{1}{2} W_t(X)\cdot\phi_t.
\ee
For simplicity we assume from now on that $\nabla^C_\nu X = 0$. Note that this implies in particular that $\nabla_\nu^C(W_t(X))=(\nabla_\nu^CW_t)(X)$. Since $\nabla_\nu^C\phi_t=0$ we get
\begin{align*}
\left.\nabla^C_\nu\nabla^{g_t}_X\phi_t\,\right|_{t=0} &= \left. \nabla^C_\nu\big(\nabla^C_X\phi_t+\tfrac{1}{2}W_t(X)\cdot\phi_t\big)\,\right|_{t=0}\\
&= R^C(\nu,X)\phi+\nabla^C_{[\nu,X]}\phi+\tfrac{1}{2}(\nabla_\nu^CW)(X)\cdot\phi,
\end{align*}
where $R^C$ is the curvature operator acting on $\Gamma(\Sigma C)$, cf.\ equation~\eqref{curvatureaction}. We investigate the first term of the bottom line. With $e_0=\nu$ we obtain
\ben
R^C(\nu,X)\phi=\tfrac{1}{2}\sum_{0\leq i<j\leq n}( R^C(\nu,X)e_i,e_j)_g\, e_i\ast e_j\ast\phi
\ee
and further
\begin{align*}
&\tfrac{1}{2}\sum_{0\leq i<j\leq n}( R^C(\nu,X)e_i,e_j)_g\, e_i\ast e_j\ast\phi\\
=&-\tfrac{1}{2}\sum_{1\leq i<j\leq n}( R^C(e_i,e_j)X,\nu)_g\, e_i\ast e_j\ast\phi-\tfrac{1}{2}\sum_{j=1}^n( R^C(X,\nu)\nu,e_j)_g\,\nu\ast e_j\ast\varphi\\
=&-\tfrac{1}{2}\sum_{1\leq i<j\leq n}( R^C(e_i,e_j)X,\nu)_g\, e_i\cdot e_j\cdot\phi-\tfrac{1}{2}\sum_{j=1}^n( R^C(X,\nu)\nu,e_j)_g\,e_j\cdot\varphi\\
=&-\tfrac{1}{2}\sum_{1\leq i<j\leq n}( R^C(e_i,e_j)X,\nu)_g\, e_i\cdot e_j\cdot\phi+\tfrac{1}{2}\sum_{j=1}^n( W^2(X)-(\nabla^C_\nu W)(X),e_j)_g\,e_j\cdot\varphi\\
=&-\tfrac{1}{2}\sum_{1\leq i<j\leq n}( R^C(e_i,e_j)X,\nu)_g\, e_i\cdot e_j\cdot\phi+\tfrac{1}{2}W^2(X)\cdot\varphi-\tfrac{1}{2}\bigl((\nabla^C_\nu W)(X)\bigr)\cdot\varphi,
\end{align*}
where we have used the Riccati equation $(\nabla^C_\nu W_t)(X)=R^C(X,\nu)\nu+W_t^2(X)$. For the second term we observe that $[\nu,X]=-\nabla^C_X\nu=W(X)$. Using the relation $\nabla^C_{W(X)}\varphi = \nabla^g_{W(X)}\varphi - \frac{1}{2}W^2(X)\cdot\phi$, which follows from equation \eqref{cov_der}, and (4.3) from~\cite{bgm05} we finally obtain
\beq\label{nablaznablat}
\left.\nabla^C_\nu\nabla_X^{g_t}\phi_t\,\right|_{t=0} = \nabla^g_{W(X)}\varphi+\tfrac{1}{4}\sum_{1\leq i<j\leq n}\big((\nabla^g_{e_i}\dot g)(e_j,X)-(\nabla^g_{e_j}\dot g)(e_i,X)\big)e_i\cdot e_j\cdot\varphi.
\ee
Since $W(X)=-\dot g(X,\cdot)/2$ by (4.1) in~\cite{bgm05}, upon substituting $X=e_k$ the first term of \eqref{nablaznablat} contributes
\ben
-\frac{1}{2}\int_M\sum_{k=1}^n\dot g(e_k,e_j)\<\nabla^g_{e_j}\phi,\nabla^g_{e_k}\phi\>\,dv^g=-\frac{1}{2}\int_M (\dot g,\<\nabla^g\phi\otimes\nabla^g\phi\>)_g \,dv^g
\ee
to the integral in~\eqref{emtensor} (with $( \cdot \,, \cdot )_g$ denoting the pointwise inner product on $\odot^2T^*\!M$ induced by $g$). Further, contracting the second term of \eqref{nablaznablat} with $\nabla^g_{e_k}\phi$ for $X=e_k$, and taking the sum over $k$ gives
\begin{align*}
&\tfrac{1}{4}\sum_{k=1}^n\sum_{i,j=1,i\neq j}^n(\nabla^g_{e_i}\dot g)(e_j,e_k)\<e_i\cdot e_j\cdot\phi,\nabla^g_{e_k}\phi\>\\ 
=&\tfrac{1}{4}\sum_{i,j,k=1}^n(\nabla^g_{e_i}\dot g)(e_j,e_k)\<e_i\wedge e_j\cdot\phi,\nabla^g_{e_k}\phi\>\\
=&\tfrac{1}{4}(\nabla^g\dot g,T_{g,\phi})_g.
\end{align*}
Finally, a pointwise computation using a local basis with $\nabla^g_{e_j}e_k=0$ in the given point yields
\begin{align*}
(\nabla^g\dot g,T_{g,\phi})_g&= \sum_{k=1}^n (\nabla^g_{e_k}\dot g,T_{g,\phi}(e_k,\cdot,\cdot))_g\\
&=\sum_{k=1}^ne_k (\dot g,T_{g,\phi}(e_k,\cdot,\cdot))_g-\sum_{k=1}^n(\dot g,(\nabla^g_{e_k}T_{g,\phi})(e_k,\cdot,\cdot))_g.
\end{align*}
The first sum vanishes is the divergence of a vector field, and thus vanishes 
after integration. The second sum is just the pointwise inner product of 
$\dot g$ with the divergence of $T_{g,\phi}$. We obtain
  $$\frac14 \int_M (\nabla^g\dot g,T_{g,\phi})_g\,dv^g=\frac14 \int_M (\dot g, \div T_{g,\phi})_g\,dv^g,$$
and the assertion follows.
\end{proof}

In the following we determine the critical points of $\mc E$, i.e.\ pairs $(g,\phi)\in\mc N$ satisfying the system of Euler--Lagrange equations
\beq\label{critpoint}
\begin{array}{l}
|\nabla^g\phi|^2_gg+\div_g T_{g,\phi}-2\<\nabla^g\phi\otimes\nabla^g\phi\>=0,\\[5pt]
\nabla^{g\ast}\nabla^g\phi =|\nabla^g\phi|^2_g\phi.
\end{array}
\ee
Integrating the trace of the first equation we obtain
\ben
0=\int_M\tr_g\big(|\nabla^g\phi|_g^2g+\div_g T_{g,\phi}-2\<\nabla^g\phi\otimes\nabla^g\phi\>\big) \, dv^g=(n-2)\|\nabla^g\phi\|_g^2,
\ee
since the integral of $\tr_g \div_g T_{g,\phi}$ vanishes by the divergence theorem. Hence $\nabla^g\phi=0$ for a critical point $(g,\phi)$ if $n\geq3$.

\medskip

In addition, we can impose the (global) constraint $\vol(M,g):=\int_M dv^g=1$. For a constrained critical point $(g,\phi)$ the metric part of the gradient must be orthogonal to $T_g \mc M_1$, where $\mc M_1$ is the space of unit volume metrics. In particular, $T_g \mc M_1 = \{ h\in \Gamma(\odot^2T^*\!M): \int_M (h, g)_g \, dv^g =0\}$. Hence the right hand side of the first equation in \eqref{critpoint} equals a constant multiple of the metric, i.e.\
\ben
|\nabla^g\phi|^2_gg+\div T_{g,\phi}-2\<\nabla^g\phi\otimes\nabla^g\phi\>=cg
\ee
for some $c \in \R$. Reasoning as above we obtain $c= \frac{n-2}{n} \| \nabla^g\phi\|^2_g$, i.e.~$c=0$ if $n=2$ and $c \geq 0$ if $n\geq 3$. In particular, if $n=2$ a constrained critical point is already a critical point. Summarising, we obtain

\begin{cor}
Let $n=\dim M\geq3$. Then the following holds:

\smallskip

\noindent(i) $(g,\phi)\in\mc N$ is critical for $\mc{E}$ if and only if $\nabla^g\phi=0$, i.e.\ $\phi$ is parallel with respect to $g$. In particular, the metric $g$ is Ricci flat. Furthermore, any critical point is an absolute minimiser of $\mc{E}$.

\smallskip

\noindent(ii) $(g,\phi)\in\mc N$ is critical for $\mc{E}$ subject to the constraint $\vol(M,g)=1$ if and only if
\ben
\begin{array}{l}
|\nabla^g\phi|^2_gg+\mr{div}_g T_{g,\phi}-2\<\nabla^g\phi\otimes\nabla^g\phi\>=cg,\\[5pt]
\nabla^{g\ast}\nabla^g\phi =|\nabla^g\phi|^2_g\phi
\end{array}
\ee
for some constant $c \geq0$. (A pair $(g,\varphi) \in \mc N$ satisfying this system of equations will henceforth be called a {\em spinor soliton}.)

\smallskip

\noindent If $n=2$, then $(g,\phi) \in \mc N$ is a genuine critical point for $\mc E$ with $\vol(M,g)=1$ if and only if it is a constrained critical point.  
\end{cor}

\begin{example*}
(i) A parallel spinor not only implies that the underlying metric is Ricci-flat, but also that the holonomy is a proper subgroup of $\SO(n)$. The converse is also true for suitable choices of a spin structure~\cite{mose00},~\cite{wa89}. For example, in dimension~$4$ and~$6$ a parallel spinor forces the underlying Riemannian manifold to be Calabi-Yau (i.e.\ the holonomy is contained in~$\SU(2)$ or~$\SU(3)$), while in dimension~$7$ and~$8$ the holonomy is contained in $\Gt$ or $\Spin(7)$. By Yau's solution of the Calabi conjecture~\cite{ya78} and the work of Joyce on holonomy~$\Gt$- and $\Spin(7)$-manifolds (see for instance his book~\cite{jo00}), compact examples, though not in an explicit manner, exist in abundance. 

\smallskip

\noindent(ii) A {\em Killing spinor} is defined to be a spinor satisfying the equation
\ben
\nabla^g_X\phi=\lambda X\cdot\phi
\ee
for all $X \in \Gamma(TM)$ and some fixed $\lambda\in\R$. As a consequence of this equation $g$ must be necessarily Einstein (see for instance~\cite{bfgk91}). A Killing spinor for a metric with $\vol(M,g)=1$ is a constrained critical point for $\mc{E}$. Indeed, by the Killing equation
\ben
\nabla^{g\ast}\nabla^g\phi=n\lambda^2\phi = |\nabla^g\phi|^2_g \phi,
\ee
so that the second equation is satisfied. The symmetrisation of $\<e_i\cdot e_j\cdot\phi,\nabla^g_{e_k}\phi\>$ is zero, for $\<X\cdot\phi,\phi\>=0$. Furthermore, $\langle\nabla^g\phi\otimes\nabla^g\phi\rangle=\lambda^2g$ (again using the Killing equation), whence
\ben
|\nabla^g\phi|^2_gg+\div_g T_{g,\phi}-2\<\nabla^g\phi\otimes\nabla^g\phi\>=(n-2)\lambda^2g.
\ee
In particular, if $n=2$ a Killing spinor for a metric with $\vol(M,g)=1$ is a genuine critical point for $\mc E$.

\smallskip

\noindent(iii) More generally a {\em twistor spinor} is a spinor in the kernel of the {\em Penrose}- or {\em twistor operator} $P_g:\Gamma(\Sigma_gM)\to\Gamma(\ker\mu)$ defined by $P_g\phi=p_{\ker\mu}\circ\nabla^g\phi$, where $p_{\ker\mu}:T^*M\otimes\Sigma_gM\to\ker\mu$ denotes projection on the kernel of Clifford multiplication. In particular, any Killing spinor is a twistor spinor. 

\end{example*}

\begin{prop}
If $\phi$ is a twistor spinor of unit norm on $(M^n,g)$ with $n\geq3$, then it is a spinor soliton. Furthermore, the scalar curvature $\scal^g$ is constant.
\end{prop}

\begin{proof}
By~\cite[Theorem I.3]{bfgk91}, $P_g\phi=0$ implies
\ben
D_g^2\phi=\frac{n}{4(n-1)}\scal^g\phi
\ee
such that using the Weitzenb\"ock formula for $D^2_g$
\ben
\nabla^{g*}\nabla^g\varphi = D_g^2\phi-\frac{\scal^g}{4}\phi = \frac{\scal^g}{4(n-1)}\varphi.
\ee
Hence, if $(g,\phi)\in\mc N$ is a twistor spinor of unit norm, then
\ben
|\nabla^g\phi|^2=(\nabla^{g*}\nabla^g\phi,\phi)-\div_g (\nabla^g\varphi,\varphi) = \frac{\scal^g}{4(n-1)}
\ee
using $(\nabla^g\phi,\phi)=\frac 12 d (\varphi,\varphi)=0$. It follows that $\nabla^{g*}\nabla^g\phi=|\nabla^g\phi|^2\phi$, i.e.\ $Q_2(g,\phi)=0$. To compute $Q_1(g,\phi)$ we recall that $\psi=X\cdot\nabla^g_X\phi$ is unambigously defined for a unit vector field $X$~\cite[Theorem I.2 (4)]{bfgk91}. Hence for an orthonormal basis $e_1,\ldots,e_n$, we find $(\nabla^g_{e_i}\phi,\nabla^g_{e_j}\phi)=(e_i\cdot\psi,e_j\cdot\psi)$ which vanishes unless $i=j$ when it equals $|\nabla^g_{e_i}\phi|^2=|\psi|^2$. The contribution of $\langle\nabla^g\phi\otimes\nabla^g\phi\rangle$ is therefore just $|\nabla^g\phi|^2g/n$. It remains to compute the divergence term. If $\tilde T$ denotes the unsymmetrised $(3,0)$-tensor we have
\ben
\tilde T_{ijk}=\langle e_i\cdot e_j\cdot\phi,\nabla_{e_k}\phi\rangle=\langle e_i\cdot e_j\cdot e_k\cdot\phi,\psi\rangle.
\ee
Since $\langle X\cdot\phi,\psi\rangle=\langle\phi,\nabla_X\phi\rangle$ for any unit vector $X$, the $(3,0)$-tensor is totally skew, hence $T_{g,\phi}=0$. Consequently,
\ben
Q_1(g,\phi)=\frac{2-n}{4n}|\nabla^g\phi|^2g.
\ee
Now by the Bianchi identity~\ref{bianchi}, the divergence $\delta_gQ_1(g,\phi)$ vanishes so that $|\nabla^g\phi|^2$ is constant. In particular, $(g,\varphi)$ will be a spinor solition with $c=\frac{2-n}{4n}|\nabla^g\phi|^2$ and constant scalar curvature $\scal^g=4(n-1)|\nabla^g\phi|^2$.
\end{proof}
%
\subsection{The second variation}\label{second_var}
%
Next we investigate conditions for $\mc E$ to be Morse-Bott, i.e.\ the critical set of $\mc E$ forms a smooth manifold whose tangent bundle is precisely the kernel of the Hessian of $\mc E$ (seen as an endomorphism via the $L^2$-metric). Towards this end, we calculate the second variation of $\mc E$ at a critical point.

\medskip

First, we linearise the spinor connection $\nabla^g \phi$ as a function of $g$ and $\phi$. More formally, for $X \in \Gamma(TM)$ consider the map
\ben
K_X:\mc F\to\mc F, \quad (g,\phi) \mapsto (g, \nabla_X^g \phi)
\ee
and decompose the tangent space $T_{(g,\phi)}\mc F = T_g \mc M \oplus \mc F_g$ as above.

\begin{lemma}\label{linnab}
Let $(g,\phi) \in \mc F$. Then $D_{(g,\phi)} K_X : T_{(g,\phi)}\mc F  \rightarrow T_{(g,\nabla^g_X\phi)} \mc F$ is given by
\beq\label{variationk}
(D_{(g,\phi)} K_X)(\dot g, \dot \phi) = (\dot g, \tfrac{1}{4} \sum_{i \neq j} (\nabla^g_{e_i} \dot g)(X,e_j) e_i \cdot e_j \cdot \phi + \nabla_X^g \dot \phi).
\ee
\end{lemma}

\begin{rem*}\label{wang}
Let $\alpha_t$ be a smooth curve of orientation preserving automorphisms of $TM$ which are symmetric with respect to $g_0=g$ and such that $\alpha_0=\id_{TM}$. If $g_t$ is the induced curve of metrics defined by $g_t(X,Y)=g(\alpha_t^{-1}X,\alpha_t^{-1}Y)$, then $\dot g(X,Y)=-2g(\dot\alpha X,Y)$. Substituting this into~\eqref{variationk} then gives Wang's formula in~\cite[Proposition 1.5]{wa91}.
\end{rem*}

\begin{proof}
Again the vertical variation is the easy part for which we obtain
\ben
(D_{(g,\phi)}K_X)(0,\dot \phi) = (0,\nabla_X^g \dot \phi)
\ee
with $\dot \phi \in \Gamma(\Sigma_gM)$.
Let now $g_t$ be a path of metrics with $g_0=g$ and $\left.\frac{d}{dt}\right\vert_{t=0} g_t = \dot g$. Let $g_t^{hor}$ be the horizontal lift of $g_t$. We may assume without loss of generality that $g_t$ is a linear path, i.e.\ $g_t = g_0 + t \dot g$. Then $g_t^{hor}=(g_t,\phi_t)$ with $\phi_t = \hat{B}_{g_t}^g \phi$.
Recall that if $\phi$ is locally expressed as $[\tilde b, \tilde \phi]$, then $\hat{B}_{g_t}^g \phi$ is locally expressed as $[\tilde{B}_{g_t}^g\tilde b, \tilde \phi]$. With $e_{i,t} = B^g_{g_t}(e_i)$ the local spin frame $\tilde B^g_{g_t}\tilde b$ covers the local orthonormal frame $B^g_{g_t}b=(e_{1,t},\ldots,e_{n,t})$ and we compute locally
\begin{align*}
(D_{(g,\phi)}K_X)(\dot g,0) =& \left.\tfrac{d}{dt}\right\vert_{t=0} \nabla_X^{g_t}\phi_t\\
=& \left.\tfrac{d}{dt}\right\vert_{t=0} [ \tilde{B}^g_{g_t}\tilde b, X \tilde \phi + \tfrac{1}{2} \sum_{i < j} g_t(\nabla_X^{g_t}e_{i,t}, e_{j,t}) E_i \cdot E_j \cdot \tilde \phi]\\
=&  [\left.\tfrac{d}{dt}\right\vert_{t=0} \tilde{B}^g_{g_t}\tilde b, X \tilde \phi + \tfrac{1}{2} \sum_{i < j} g(\nabla_X^ge_i, e_j) E_i \cdot E_j \cdot \tilde \phi]\\
& + [\,\tilde b, \left.\tfrac{d}{dt}\right\vert_{t=0} \tfrac{1}{2} \sum_{i < j} g_t(\nabla_X^{g_t}e_{i,t}, e_{j,t}) E_i \cdot E_j \cdot \tilde \phi]\\
=& \left.\tfrac{d}{dt}\right\vert_{t=0} \hat B^g_{g_t} \nabla_X^g \phi  + \tfrac{1}{2} \sum_{i < j} \left.\tfrac{d}{dt}\right\vert_{t=0} g_t(\nabla_X^{g_t}e_{i,t}, e_{j,t}) e_i \cdot e_j \cdot \phi
\end{align*}
using equation~\eqref{spinor_conn}. The first term in this sum gives
\ben
\left.\tfrac{d}{dt}\right\vert_{t=0} \hat B^g_{g_t} \nabla_X^g \phi= (\dot g, 0 ) \in  T_{(g,\nabla^g_X\phi)} \mc F.
\ee
For the second term we observe that
\ben
\left.\tfrac{d}{dt}\right\vert_{t=0} g_t(\nabla_X^{g_t}e_{i,t}, e_{j,t}) = g( \left.\tfrac{d}{dt}\right\vert_{t=0} \nabla_X^{g_t}e_i, e_j) + g(\nabla^g_X \left.\tfrac{d}{dt}\right\vert_{t=0} e_{i,t}, e_j)
\ee
if we compute at a point $x\in M$ with an orthonormal frame such that $\nabla^g e_i (x) = 0$. Using the generalised cylinder calculus as in the proof of Theorem~\ref{critpointthm} we get
\ben
\left.\tfrac{d}{dt}\right\vert_{t=0} e_{i,t} = \left.[\nu,e_{i,t}] \,\right\vert_{t=0} = -\left.\nabla^C_{e_{i,t}} \nu \,\right\vert_{t=0} = W(e_i) = - \tfrac{1}{2}\dot g(e_i, \cdot),
\ee
since $\nabla^C_\nu e_{i,t} = 0$ and $W(X)=-\dot g(X,\cdot)/2$ by (14) in~\cite{bgm05}. Consequently,
\ben
g(\nabla^g_X \left.\tfrac{d}{dt}\right\vert_{t=0} e_{i,t}, e_j)=-\tfrac{1}{2} (\nabla_X^g \dot g)(e_i,e_j). 
\ee
Furthermore, by Theorem 1.174 (a) in~\cite{be87} we have
\ben
g( \left.\tfrac{d}{dt}\right\vert_{t=0} \nabla_X^{g_t}e_i, e_j) = \tfrac{1}{2} \bigl( (\nabla^g_X \dot g) (e_i,e_j) + ( \nabla^g_{e_i} \dot g)(X,e_j) - (\nabla^g_{e_j} \dot g)(e_i,X) \bigr),
\ee
whence
\ben 
\tfrac{1}{2} \sum_{i < j} \left.\tfrac{d}{dt}\right\vert_{t=0} g_t(\nabla_X^{g_t}e_{i,t}, e_{j,t}) e_i \cdot e_j \cdot \phi = \tfrac{1}{4} \sum_{i\neq j}  ( \nabla^g_{e_i} \dot g)(X,e_j) e_i \cdot e_j \cdot \phi.
\ee
This is the asserted formula.
\end{proof}

For $(g,\phi) \in \mc F$ let $\kappa_{g,\phi}: T_{(g,\phi)}\mc F \rightarrow \Gamma(T^*\!M \otimes \Sigma_gM)$ be the map defined by
\ben
\kappa_{g,\phi}(\dot g,\dot \phi) :=  \tfrac{1}{4} \sum_{i \neq j} (\nabla^g_{e_i} \dot g)(\cdot\,,e_j) e_i \cdot e_j \cdot \phi + \nabla^g \dot \phi,
\ee
so that $(D_{(g,\phi)} K_X)(\dot g, \dot \phi) = (\dot g, \kappa_{g,\phi}(\dot g,\dot \phi)(X))$. Take a smooth path $(g_t,\phi_t)$ with $(g_0,\phi_0)=(g,\phi)$ and write $\nabla^{g_t}\phi_t=\sum_ie_i^\flat\otimes\nabla^{g_t}_{e_i}\phi_t$ for a local orthonormal frame $\{e_i\}$ with respect to $g$. Then we have
\begin{align}\label{der_norm_sq}
\left. \tfrac{d}{dt} \right|_{t=0} |\nabla^{g_t}\phi_t|^2_{g_t} &= \left. \tfrac{d}{dt} \right|_{t=0} \sum_{i,j} g_t(e_i^\flat,e_j^\flat) \langle \nabla_{e_i}^{g_t} \phi_t, \nabla_{e_j}^{g_t} \phi_t\rangle\notag\\
&= -\sum_{i,j} \dot g (e_i,e_j) \langle \nabla_{e_i}^g\phi, \nabla_{e_j}^g \phi \rangle + 2 \sum_{i}\langle \kappa_{g,\phi}(\dot g, \dot \phi) (e_i) , \nabla_{e_i}^g \phi \rangle\notag\\
&= -(\dot g, \langle \nabla^g \phi \otimes \nabla^g \phi \rangle )_g + 2 (\kappa_{g,\phi}(\dot g, \dot \phi), \nabla^g \phi)_g 
\end{align}
with $\dot g = \left.\frac{d}{dt}\right|_{t=0} g_t$ and  $\dot \phi = \left.\frac{d}{dt}\right|_{t=0} \phi_t$ (recall that $\left.\frac{d}{dt}\right|_{t=0}g_t(v^\flat,w^\flat) = - \dot g(v,w)$ for any $v,w \in T_xM$).

\begin{prop}[\bf First variation revisited]
Let $(g_t,\phi_t)$ be a smooth path in $\mc N$ with $(g_0,\phi_0)=(g,\phi)$. Then
\begin{align*}
\left. \tfrac{d}{dt} \right|_{t=0} \mc E(g_t,\phi_t)=& \int_M (\dot g, \tfrac{1}{4} |\nabla^g \phi|_g^2g + \tfrac{1}{4} \div_g T_{g,\phi} - \tfrac{1}{2} \langle \nabla^g \phi \otimes \nabla^g \phi \rangle )_g \, dv^g\\
& + \int_M (\dot \phi, \nabla^{g*} \nabla^g \phi)_g \, dv^g
\end{align*}
with $\dot g = \left.\frac{d}{dt}\right|_{t=0} g_t$ and  $\dot \phi = \left.\frac{d}{dt}\right|_{t=0} \phi_t$.
\end{prop}

\begin{proof}
Using equation \eqref{der_norm_sq} we compute 
\begin{align*}
&\left.\tfrac{d}{dt}\right\vert_{t=0} \mc{E}(g_t,\phi_t)\\
 =& \tfrac{1}{2}\int_M|\nabla^g\phi|_g^2\,        \left.\tfrac{d}{dt}\right\vert_{t=0}dv^{g_t} + \tfrac{1}{2}\int_M\left.\tfrac{d}{dt}\right\vert_{t=0} \vert\nabla^{g_t}\phi_t|_{g_t}^2 \, dv^g\\
=& \tfrac{1}{4}\int_M|\nabla^g\phi|_g^2\,\tr_g\dot g \, dv^g  - \tfrac{1}{2} \int_M (\dot g, \langle \nabla^g \phi \otimes \nabla^g \phi \rangle )_g \, dv^g + \int_M (\kappa_{g,\phi}(\dot g, \dot \phi), \nabla^g \phi)_g \, dv^g.
\end{align*}
Now 
\ben
\tfrac{1}{4}\int_M|\nabla^g\phi|_g^2\,\tr_g\dot g \, dv^g = \tfrac{1}{4} \int_M |\nabla^g \phi|^2_g \,(\dot g, g)_g \, dv^g
\ee
and
\begin{align*}
(\kappa_{g,\phi}(\dot g, \dot \phi), \nabla^g \phi)_g &= \tfrac{1}{4} \sum_{ i \neq j} \sum_k \langle (\nabla_{e_i}^g \dot g)( e_k, e_j) e_i \cdot e_j \cdot \phi, \nabla^g_{e_k} \phi\rangle + \sum_k \langle \nabla^g_{e_k}\dot \phi, \nabla^g_{e_k} \phi \rangle\\
&= \tfrac{1}{4} ( \nabla^g \dot g, T_{g,\phi})_g + (\nabla^g \dot \phi, \nabla^g \phi)_g. 
\end{align*}
Since
\ben
\tfrac{1}{4} \int_M  ( \nabla^g \dot g, T_{g,\phi})_g \, dv^g = \tfrac{1}{4} \int_M (\dot g, \div_g T_{g,\phi})_g \, dv^g
\ee
and
\ben
\int_M  (\nabla^g \dot \phi, \nabla^g \phi)_g \, dv^g = \int_M (\dot \phi , \nabla^{g*} \nabla^g \phi)_g \, dv^g
\ee
the result follows.
\end{proof}

\begin{prop}[\bf Second variation]
Let $(g,\phi) \in \mc N$ with $\nabla^g\phi=0$ and let $(g_t,\phi_t)$ be a smooth path in $\mc N$ with $(g_0,\phi_0)=(g,\phi)$. Then
\ben
\left.\tfrac{d^2}{dt^2}\right\vert_{t=0} \mc E(g_t,\phi_t) = \int_M | \kappa_{g,\phi} (\dot g, \dot \phi)|_g^2 \, dv^g = \int_M (\kappa_{g,\phi}^* \kappa_{g,\phi} (\dot g, \dot \phi),(\dot g, \dot \phi))_g \, dv^g \geq 0
\ee
with equality if and only if $\kappa_{g,\phi} (\dot g, \dot \phi)=0$.
\end{prop}

\begin{proof}
Computing as before we get
\begin{align*}
\left.\tfrac{d^2}{dt^2}\right\vert_{t=0} \mc{E}(g_t,\phi_t) &= \left.\tfrac{d}{dt}\right\vert_{t=0} \left ( \tfrac{1}{4}\int_M|\nabla^{g_t}\phi_t|_{g_t}^2\,\tr_{g_t}\dot g_t \, dv^{g_t}  + \tfrac{1}{2} \int_M\dot g_t(\nabla^{g_t} \phi_t, \nabla^{g_t} \phi_t) \, dv^{g_t} \right.\\
&\quad \quad \quad \quad \left. + \int_M (\kappa_{g_t,\phi_t}(\dot g_t, \dot \phi_t), \nabla^{g_t} \phi_t)_{g_t} \, dv^{g_t} \right)\\
& =  \int_M | \kappa_{g,\phi} (\dot g, \dot \phi)|_g^2 \, dv^g
\end{align*}
since any $\nabla^{g_t} \phi_t$-term which is not differentiated vanishes by assumption when evaluated at $t=0$. For instance, $\sum_i\left.\tfrac{d}{dt}\right\vert_{t=0}e_{i,t}\otimes\nabla^{g_t}_{e_{i,t}}\phi_t=\sum_i e_i\otimes\left.\tfrac{d}{dt}\right\vert_{t=0}\nabla^{g_t}_{e_{i,t}}\phi_t=\kappa_{g,\phi} (\dot g, \dot \phi)$. The result follows.
\end{proof}

The linearisation $L_{g,\phi}:=D_{(g,\phi)} Q:T_{(g,\phi)}\mc N\to T_{(g,\phi)}\mc N$ of the negative gradient of $\mc E$ at a critical point is just the negative Hessian of $\mc E$ regarded as an endomorphism via the $L^2$-metric.

\begin{cor}\label{Lkernel}
If $(g,\phi) \in \mc N$ satisfies $\na^g\phi=0$, then
\ben
L_{g,\phi} = - \pi^{T\mc N} \circ \kappa_{g,\phi}^* \kappa_{g,\phi}.
\ee
In particular, $L_{g,\phi}$ is formally selfadjoint and non-positive in the sense that 
\ben
\llangle L_{g,\phi} (\dot g, \dot \phi), (\dot g, \dot \phi) \rrangle_g \leq 0
\ee 
for all $\dot g \in \Gamma(\odot^2 T^*\!M)$ and $\dot \phi \in \Gamma(\Sigma_g M)$.
Furthermore, $\ker L_{g,\phi} = \ker \kappa_{g,\phi}$, which is the space of infinitesimal deformations preserving a parallel spinor.
\end{cor}

Let $\mr{Crit}(\mc E)$ denote the critical set of $\mc E$ and let $\mc R$ be the space of Ricci-flat metrics. In the following we assume $n=\dim M \geq 3$, which implies that $(g,\phi)\in\mc N$ is in $\mr{Crit}(\mc E)$ if and only if $\nabla^g\phi=0$. In particular, $g\in\mc R$ for $(g,\phi)\in\mr{Crit}(\mc E)$. 

\medskip

Recall that a Riemannian manifold $(M,g)$ is {\em irreducible} if its universal Riemannian cover is not isometric to a Riemannian product. A critical point will be called irreducible if the underlying Riemannian manifold is irreducible.

\begin{thm}\label{integkerL}
Let $n=\dim \geq 4$ and $(\bar g,\bar\phi)$ be an irreducible critical point. If either

\smallskip

(i) $M$ is simply-connected, or

\smallskip

(ii) $n$ is not divisible by four, or

\smallskip

(iii) $n=4$ or $12$,

then there exists a smooth neighbourhood $\mc V$ of $(\bar g,\bar\phi)$ inside $\mr{Crit}(\mc E)$ such that $T_{(\bar g,\bar\phi)}\mc V=\ker L_{\bar g,\bar\phi}$.
\end{thm}

\begin{rem*}
In dimension $3$ a metric with a parallel spinor is necessarily flat and in fact only flat tori can carry such metrics~\cite{pf00}.
\end{rem*}

Before we actually prove this theorem we first remark that by Corollary~\ref{Lkernel} and Remark~\ref{wang}, $\ker L_{g,\phi}=\ker d\mc L^0$ for $(g,\phi)\in\mr{Crit}(\mc E)$, where $d\mc L^0$ is defined as in~\cite[Proposition 2.2]{wa91}. In particular, we get
\beq\label{lkernel}
\ker L_{g,\phi}\cap\delta_g^{-1}(0)\cap\tr^{-1}(0)=\{(\dot g,\dot\phi)\,|\,\delta_g\dot g=0,\,\tr_g\dot g=0,\,\nabla^g\dot\phi=0,\,\mc D_g\Psi_{\dot g,\phi}=0\}.
\ee
The twisted Dirac operator $\mc D_g:\Gamma(T^*\!M\otimes\Sigma_gM)\to\Gamma(T^*\!M\otimes\Sigma_gM)$ is defined locally by $\mc D_g(\alpha\otimes\phi)=\sum e_k\cdot\alpha\otimes\nabla^g_{e_k}\phi+\alpha\otimes D_g\phi$, and is evaluated on $\Psi_{\dot g,\phi}$, the spinor-valued $1$-form given by $\Psi_{\dot g,\phi}(X)=\dot g(X)\cdot\phi$. By~\cite[Proposition 2.10]{wa91},~\cite[Proposition 2.4]{dww05}, $\mc D_g\Psi_{\dot g,\phi}=0$ if and only if $\Delta_L\dot g=0$. Here, $\Delta_L$ denotes the Lichnerowicz Laplacian defined with respect to $g$. Its kernel consists of the transverse traceless infinitesimal deformations of $g$~\cite[Theorem 1.174 (d)]{be87}. It follows that the space on the right hand side of~\eqref{lkernel} can be identified with $\ker\Delta_L\times\{g\mbox{-parallel spinors}\}$. Smoothness at $(g,\phi)\in\mr{Crit}(\mc E)$ will therefore follow from two properties. For the $\diff_0(M)$-action on $\mc R$ we need a smooth slice through $g$ which gives rise to a smooth open neighbourhood $\mc U\subset\mc R$ of $g$. Then $\ker L_{g,\phi}$ can be identified with $T_g\mc U\times\{g\mbox{-parallel spinors}\}$. Secondly, $D^2_g=\nabla^{g*}\nabla^g+\mr{scal}^g/4$ in virtue of the Weitzenb\"ock formula, so that any $g$-harmonic spinor must be parallel for $g\in\mc R$. We therefore consider the restriction of the fibre bundle $\mc F\to\mc M$ to $\mc U$ together with the ``universal'' Dirac operator $D:\mc F|_{\mc U}\to\mc F|_{\mc U}$ which sends $\Phi=(g,\phi)$ to  $(g,D_g\phi)$. Formally, this is a smooth family of elliptic operators. 
If $\dim \ker D^g$ is constant on $\mc U$, then $\bigcup_{g \in \mc U} \ker D_g$ forms a smooth vector bundle over $\mc U$ and we may take $\mc V$ to be the total space of the associated unit sphere bundle.

\begin{proof}[Proof of Theorem~\ref{integkerL}]

Under the assumption of the theorem it follows from~\cite{wa89} for case (i) and~\cite{wa95} (building on McInnes work~\cite{mc91}) for case (ii) and (iii) that the holonomy must be equal to either $\Gt$, $\Spin(7)$, $\SU(m)$ or $\Sp(k)$ depending on $n$ and the number of linearly independent spinors. Consequently, $g$ defines a {\em torsionfree $G$-metric} in the sense of~\cite{no10}. As shown there, these metrics define an open smooth subset of $\mc R$ as a consequence of Goto's unobstructedness theorem~\cite{go04}.

\medskip

To settle the local constancy we shall proceed case by case, using the following facts. First assume $n=\dim M$ to be even so that we have a decomposition $\mc F_g=\mc F_{g+}\oplus\mc F_{g-}$ into positive and negative spinors. The index of $D^\pm_g=D_g|_{\mc F_{g\pm}}$ is given by  $\ind D^{\pm}_g=\pm(\delta^+_g-\delta^-_g)$, where $\delta^{\pm}_g=\dim\ker D^{\pm}_g$. The $\delta^{\pm}_g$ are upper semi-continuous in $g$ so they can only decrease simultaneously in a neighbourhood of $\bar g$, for the index is locally constant. In particular, $\delta^{\pm}_g$ is locally constant if either $\delta^+_{\bar g}$ or $\delta^-_{\bar g}$ vanishes. On the other hand $\bar g$ is irreducible by assumption, and so are metrics close to $\bar g$ as this is an open condition. Consequently,  $\delta^{\pm}_{\bar g}$ can be read off from the classification results of Wang~\cite{wa89} if $M$ is simply-connected and~\cite{wa95} for the non-simply connected case. If $M$ is simply-connected there are four cases if $n$ is even: $\Spin(7)$ ($n=8$), $\SU(m)$ for $n=2m$ even or odd respectively, and finally $\Sp(m)$ for $n=4m$. The only non-trivial case where both $\delta^+_{\bar g}$ and $\delta^-_{\bar g}$ are non-zero occurs for $\SU(m)$, $m$ odd. Here $\delta_{\bar g}^+=\delta_{\bar g}^-=1$. But if $\delta^{\pm}_g$ were not constant near $\bar g$, then necessarily $\delta_g^+=\delta_g^-=0$ arbitrarily close to $\bar g$.  This cannot happen for irreducible Ricci-flat metrics of special holonomy are open in $\mc R$. Finally, if $n$ is odd, the only case which can occur is holonomy~$\Gt$. Again any Ricci-flat metric close enough to $\bar g$ must have holonomy exactly $\Gt$ so that $\dim\ker D_g=1$. The non simply-connected case follows analogously if one excludes the case of holonomy $\SU(4m)\ltimes\Z_2$ for $n=8m$ and $\Sp(k)\times\Z_d$, where $d>1$ odd divides $m+1$.
\end{proof}

\begin{rem*}
In any case it follows from Wang's stability theorem~\cite[Theorem 3.1]{wa91} and~\cite[Section 3]{wa95} in conjunction with~\cite{go04} and~\cite{no10} that the irreducible metrics with parallel spinor form an open (but not necessarily smooth) subset of $\mc R$ on which $\ker D$ is locally constant.
\end{rem*}

In low dimensions, irreducibility can be enforced by the topology of the underlying manifold. Indeed, we have the

\begin{lem}\label{irredcrit}
Let $M$ be a compact spin manifold. Assume either that

\smallskip

(i) $n=\dim M=4$, $6$ or $7$ and $M$ has finite fundamental group, or

\smallskip

(ii) $n=8$, $M$ is simply-connected and not a product of two $K3$-surfaces. 

\smallskip

Then any metric which admits a parallel spinor is irreducible.
\end{lem}

\begin{cor}\label{cas.cri.smo}
Under the above assumptions, the functional $\mc E$ is Morse--Bott, i.e.\ the critical set of $\mc E$ is a submanifold with tangent space precisely given by the kernel of the Hessian of $\mc E$.
\end{cor}

\begin{proof}[Proof of Lemma~\ref{irredcrit}] If $(M^n,g)$ admits a parallel spinor, then $M^n$ has a finite Riemannian cover of the form $(\tilde M^{n-k},\tilde g)\times(T^k,g_0)$ where $(\tilde M^{n-k},\tilde g)$ is a compact simply-connected Ricci-flat manifold \cite{chgr71},~\cite{fiwo75}. Since $\pi_1(M)$ is finite this implies that $k=0$, identifying $\tilde M$ with the universal covering of $M$. In dimensions $4,6$ and $7$ we conclude that $\tilde M$ is irreducible for otherwise its de Rham decomposition would contain a Euclidean factor on dimensional grounds. In dimension $8$, if $M$ is simply-connected and admits a parallel spinor, the holonomy group of $M$ is either $\SU(2) \times \SU(2)$, $\Sp(2)$, $\SU(4)$ or $\Spin(7)$, where $M$ is irreducible in the latter three cases and is a product of two $K3$-surfaces in the first case.
\end{proof}

\begin{rem*}
In dimension $5$ a Riemannian manifold with a parallel spinor is either flat or a mapping torus of a holomorphic isometry of a $K3$-surface \cite{frka90}.  
\end{rem*}

\begin{rem*}\label{art.in.pre}
In an article in preparation \cite{holrig} the authors prove in collaboration with
Klaus Kr\"oncke that $\dim\ker D$ is locally constant on the space of metrics admitting a nontrivial parallel spinor. In particular, if $n\geq3$ and $(\bar g,\bar\phi)$ is a critical point, then there exists a smooth neighbourhood $\mc V$ of $(\bar g,\bar\phi)$ inside $\mr{Crit}(\mc E)$ such that $T_{(\bar g,\bar\phi)}\mc V=\ker L_{\bar g,\bar\phi}$.
\end{rem*}
\section{The negative gradient flow}\label{flow}
%
%
In order to find critical points of the energy functional $\mc E$, it is natural to consider its negative gradient flow on $\mc N$.

\begin{definition}
Let $\Phi=(g,\phi)\in\mc N$. We call a solution to
\beq
\tfrac{\partial}{\partial t}\Phi_t=-\grad \cE (\Phi_t),\quad\Phi_0=\Phi\label{floweq}
\ee
the {\em spinor flow with initial condition} $\Phi$ or the {\em spinor flow} for short.
\end{definition}
 
The main goal of this section is to show that the spinor flow equation \eqref{floweq} has a unique short-time solution, i.e.\ to prove the following

\begin{thm}
Given $\Phi=(g,\phi)\in\mc N$, there exists $\epsilon>0$ and a smooth family $\Phi_t\in\mc N$ for $t\in[0,\epsilon]$ \st
\[
\tfrac{\partial}{\partial t}\Phi_t= - \grad \cE(\Phi_t),\quad\Phi_0=\Phi.
\]  
Furthermore, if $\Phi_t$ and $\Phi'_t$ are solutions to~\eqref{floweq}, then $\Phi_t=\Phi'_t$ whenever defined. Hence $\Phi_t$ is uniquely defined on a maximal time-interval $[0,T)$ for some $0<T\leq\infty$.
\end{thm}
%
\subsection{Existence}
%
In this subsection we establish the existence part. As a first step, we note the following 
\begin{lem}
The operator $Q:\mc N\to T\mc N$ is a second order, quasilinear differential operator.
\end{lem}

\begin{proof}
The assertion is a straightforward consequence of Theorem~\ref{critpointthm}. Only the rough Laplacian $\nabla^{g\ast}\nabla^g\phi$ and the divergence $\div_g T_{g,\phi}$ contribute to the highest order terms. Expressed in a local basis, these are of the form $a(\phi,\nabla\phi)\nabla_{e_j}\nabla_{e_k}$ for a smooth function $a$ depending on the $1$-jet of $\phi$.
\end{proof}

\medskip

In order to apply the standard theory of quasilinear parabolic equations (see for instance ~\cite[\S 4.4.2]{au97},~\cite[\S 7.8]{lsu67},~\cite[\S 7.1]{ta91} or~\cite[\S 4]{to06}) we recall the following definition. 
Consider a Riemannian vector bundle $(E,(\cdot\,,\cdot))$. Let $Q_t:\Gamma(E)\to\Gamma(E)$ be a time-dependent family of second order quasilinear differential operators. We say that this family is {\em strongly elliptic at $u_0$} if there exists a constant $\lambda>0$ such that the linearisation $D_{u_0}Q_0$ of $Q_0$ at $u_0$ satisfies
\beq\label{strongpara}
-(\sigma_\xi(D_{u_0}Q_0)v,v)\geq\lambda|\xi|^2|v|^2
\ee
for all $\xi\in T_x^*M$, $\xi\neq 0$, and $v\in E_x$.
If equation \eqref{strongpara} only holds with $\lambda \geq 0$, we say that $Q_t$ is {\em weakly elliptic at $u_0$}. Note that we define the principal symbol of a linear operator $L$ as
\ben
\sigma_\xi(L)v=\tfrac{i^k}{k!}L(f^ku)(x)
\ee
for any $f\in C^\infty(M)$ with $f(x)=0$, $df_x=\xi$ and $u\in\Gamma(E)$ with $u(x)=v$ which accounts for the minus sign in~\eqref{strongpara}. The corresponding flow equation
\beq\label{standeq}
\frac{\partial}{\partial t}u=Q_t(u),\quad u(0)=u_0
\ee
will be called {\em strongly parabolic at $u_0$} or {\em weakly parabolic at $u_0$} respectively. In case of strong parabolicity we have short-time existence and uniqueness for the flow equation \eqref{standeq}.

\medskip

Observe that in our case $Q = -\grad \mc E$ is a map from $\mc N=\Gamma(S(\Sigma M))$ to $T\mc N$, so that formally speaking we are not in the situation just described. However, using the natural connection of Section~\ref{dependence}, a neighbourhood of a section $\Phi = (g, \phi) \in \Gamma(S(\Sigma M))$ may be identified with a neighbourhood of the $0$-section in the vector bundle $\odot^2T^*\!M\oplus \phi^{\perp}$. More precisely, we project $\dot \phi_x \in \phi^{\perp}$ onto $S(\Sigma_gM)$ using the fiberwise exponential map and parallel translate the resulting unit spinor along the linear path $g_x + t \dot g_x$ for $\dot g_x \in \odot^2T^*\!M$. We may thus identify sections of $S(\Sigma M)$ close to $\Phi$ with sections $(\dot g, \dot \phi)$ of $\Gamma(\odot^2 T^*\!M) \oplus \Gamma(\phi^{\perp})$ accordingly. By a slight abuse of notation, we can therefore consider $Q$ as well as its linearisation $D_\Phi Q$ as an operator from $\Gamma(\odot^2 T^*\!M) \oplus \Gamma(\phi^{\perp})$ to itself.

\medskip

First we compute the principal symbol of $Q$, i.e.\ the principal symbol of its linearisation $D_\Phi Q$. Let $X \in \Gamma(TM)$ be a vector field. Recall from Lemma \ref{linnab} that the linearisation of the spinor connection $K_X: \mc F \rightarrow \mc F, (g,\phi) \mapsto (g,\nabla^g_X \phi)$ is given by
\ben
(D_{(g,\phi)} K_X)(\dot g, \dot \phi) = (\dot g, \tfrac{1}{4} \sum_{i \neq j} (\nabla^g_{e_i} \dot g)(X,e_j) e_i \cdot e_j \cdot \phi + \nabla_X^g \dot \phi).
\ee

\begin{lem}\label{symb_conn}
Let $0\neq\xi\in T^*_xM$. The principal symbol of the linearisation of the spinor connection at $\xi$ is given by
\ben
\sigma_\xi (D_{(g,\phi)} K_X)(\dot g_x, \dot \phi_x) = \bigl(0,\tfrac{i}{4} \xi \wedge \dot g_x(X, \cdot) \cdot \phi + i\xi(X)\dot \phi_x)\bigr). 
\ee
\end{lem}

\begin{proof}
Let $f$ be a smooth function on $M$ with $f(x)=0$ and $df_x=\xi$ and $\dot g$ be a symmetric 2-tensor extending $\dot g_x$, resp.\ $\dot \phi$ a spinor extending $\dot \phi_x$. Then
\begin{align*}
\sigma_\xi(D_{(g,\phi)}K_X)(\dot g_x,\dot \phi_x)  &= i(D_{(g,\phi)} K_X)(f\dot g, f \dot \phi)(x)
\\
&=  \bigl(0, \tfrac{i}{4} \sum_{i \neq j} df(e_i) \dot g(X,e_j) e_i \wedge e_j \cdot \phi+idf(X) \dot \phi\bigr)(x),
\end{align*}
whence the result.
\end{proof}

We can now proceed and compute the principal symbol of $D_{(g,\phi)} Q$. 

\begin{prop}\label{symbol_Q}
Let $0\neq\xi\in T^*_xM$ and $e_1=\xi^\sharp/|\xi|,e_2,\ldots,e_n$ be an orthonormal basis of $T_xM$. Further, define $\beta_\xi=\sum_{j=1}^n\<\xi\wedge e_j\cdot\phi,\dot\phi_x\>e_j$. Then the principal symbol of the linearisation of $Q$  at $\xi$ is given by
\ben
\sigma_\xi(D_{(g,\phi)} Q)(\dot g_x,\dot\phi_x)=\left(\begin{array}{c}\tfrac{1}{16}\bigl(-|\xi|^2\dot g_x+\xi\odot\dot g_x(\xi,\cdot)\bigr) - \frac{1}{4}\xi\odot\beta_\xi\\[5pt]-\tfrac{1}{4}\xi\wedge\dot g_x(\xi,\cdot)\cdot\phi - |\xi|^2\dot\phi_x\end{array}\right).
\ee
\end{prop}
\begin{proof}
We first notice that in order to compute the principal symbol we only need to linearise the highest order terms in $Q$, i.e.~the term $-\frac{1}{4} \div_g T_{g,\phi}$ in $Q_1$ and the term $-\nabla^{g*}\nabla^g \phi$ in $Q_2$. In the following we write $\cong$ for equality up to terms of lower order (one or zero). Firstly, beginning with $Q_2$, we get
\begin{align*}
(D_{(g,\phi)} Q_2)(\dot g,\dot\phi) &\cong D_{(g,\phi)} ((g,\phi) \mapsto 
-\nabla^{g*}\nabla^g \phi)(\dot g, \dot \phi)\\
&\cong \sum_i\nabla^g_{e_i}\bigl((D_{(g,\phi)}K_{e_i})(\dot g,\dot \phi)\bigr),
\end{align*}
and the usual symbol calculus together with Lemma~\ref{symb_conn} yields
\begin{align*}
  \sigma_\xi(D_{(g,\phi)} Q_2)(\dot g_x,\dot\phi_x)&=-\sum_{i=1}^n\xi(e_i)\bigl(\tfrac{1}{4}\xi\wedge\dot g_x(e_i,\cdot)\cdot\phi+\xi(e_i)\dot\phi_x\bigr)\\
&=-\tfrac{1}{4}\xi\wedge\dot g_x(\xi,\cdot)\cdot\phi-|\xi|^2\dot\phi_x.
\end{align*}
In a similar fashion, we obtain 
\begin{align*}
(D_{(g,\phi)} Q_1)(\dot g,\dot\phi)& \cong -\tfrac{1}{4}D_{(g,\phi)}((g,\phi) \mapsto \div_gT_{g,\phi})(\dot g,\dot\phi)\\
& \cong \tfrac{1}{4}\sum_{i=1}^n\sum_{j,k=1}^ne_i\<e_i\wedge e_j\cdot\phi,(D_{(g,\phi)}K_{e_k})(\dot g,\dot \phi)\>\,e_j\odot e_k\\
& \cong \tfrac{1}{4}\sum_{i,j,k}\<e_i\wedge e_j\cdot\phi,\nabla_{e_i}^g(D_{(g,\phi)}K_{e_k})(\dot g,\dot \phi)\>\,e_j\odot e_k,\\
\end{align*}
whence
\begin{align*}
&\sigma_\xi(D_{(g,\phi)} Q_1)(\dot g_x,\dot\phi_x)\\
=&-\tfrac{1}{4}\sum_{i,j,k}\<e_i\wedge e_j\cdot\phi,\xi(e_i)(\tfrac{1}{4}\xi\wedge\dot g_x(e_k,\cdot)\cdot\phi+\xi(e_k)\dot\phi_x)\> \,e_j \odot e_k\\
= &-\sum_{j,k}\bigl(\tfrac{1}{16}\<\xi\wedge e_j\cdot\phi,\xi\wedge\dot g_x(e_k,\cdot)\cdot\phi\>+\tfrac{1}{4}\<\xi\wedge e_j\cdot\phi,\xi(e_k)\dot\phi_x\>\bigr) \, e_j\odot e_k.
\end{align*}
The second term in the last line contributes $-\xi\odot\beta/4$. Using again the general rule $v\cdot\xi^\sharp=v\wedge\xi^\sharp-\xi(v)$ the first term gives
\begin{align*}
&-\tfrac{1}{16} \sum_{j,k}\<\xi\cdot e_j\cdot\phi+\xi(e_j)\phi,\xi\cdot\dot g_x(e_k,\cdot)\cdot\phi+\dot g_x(e_k,\xi)\phi\> \, e_j\odot e_k\\
=& -\tfrac{1}{16}\sum_{j,k} \bigl( \<\xi\cdot e_j\cdot\phi, \xi\cdot\dot g_x(e_k,\cdot)\cdot\phi \> + \<\xi\cdot e_j\cdot\phi,\dot g_x(e_k,\xi)\phi\>\\
& \qquad \quad \; + \<\xi(e_j)\phi, \xi\cdot\dot g_x(e_k,\cdot)\cdot\phi\>+\<\xi(e_j)\phi, \dot g_x(e_k,\xi)\phi\>\bigr) \,e_j\odot e_k\\
=& -\tfrac{|\phi|^2}{16}(|\xi|^2\dot g_x-\xi\odot\dot g_x(\xi,\cdot)-\xi\odot\dot g_x(\xi,\cdot)+\xi\odot\dot g_x(\xi,\cdot))\\
=&-\tfrac{|\phi|^2}{16}(|\xi|^2\dot g_x-\xi\odot\dot g_x(\xi,\cdot)).
\end{align*}
When passing from the second to the third line, we use the skew-adjointness of the Clifford multiplication and in particular the relation $\langle v \cdot \varphi, w \cdot \varphi \rangle = g(v,w)|\varphi|^2$. Since $|\phi|=1$ the desired formula follows.
\end{proof}

\begin{cor}\label{kerQ}
The principal symbol of $D_{(g,\phi)}Q$ is negative semi-definite, i.e.\
\[
g(\sigma_\xi(D_{(g,\phi)} Q) (\dot g_x, \dot \phi_x), (\dot g_x, \dot \phi_x)) \leq 0
\]
for all $\xi \in T_x^*M$, $\dot g_x \in \odot^2T^*_xM$ and $\dot \phi_x \in \phi_x^{\perp}$.  
For $\xi \in T_x^*M$ with $|\xi| =1$ the kernel of the symbol is given by
\[
\ker\sigma_\xi(D_{(g,\phi)} Q) = \{(\dot g_x,\dot\phi_x) \,|\,  \dot g_x(v,w)=0 \text{ for } v,w\perp\xi^\sharp,\, \dot \phi_x = - \tfrac{1}{4} \xi \wedge \dot g_x(\xi, \cdot) \cdot \phi \},
\]
which is an $n$-dimensional subspace of $\odot^2T_x^*M \oplus \phi_x^{\perp}$.
\end{cor}

\begin{proof}
By homogeneity of the symbol we may assume that $|\xi|=1$ to begin with. Furthermore, we decompose $\dot g_x=a\xi\otimes\xi+\xi\otimes\alpha+\alpha\otimes\xi+\gamma$, i.e.\ (using the basis of the previous proposition) $\alpha=\sum_{k=2}^n\dot g_x(\xi,e_k)e_k$ and $\gamma=\sum_{j,k=2}^n\dot g_x(e_j,e_k)e_j\otimes e_k$. Then
\begin{align*}
&g(\sigma_\xi(D_{(g,\phi)} Q) (\dot g_x, \dot \phi_x), (\dot g_x, \dot \phi_x))\\
=& -\tfrac{1}{16} |\dot g_x|^2 + \tfrac{1}{16} (\xi \odot \dot g_x( \xi, \cdot ), \dot g_x ) - \tfrac{1}{4} (\xi \odot \beta_\xi, \dot g_x ) - \tfrac{1}{4} \langle \xi \wedge \dot g_x (\xi, \cdot ) \cdot \phi, \dot \phi\rangle - |\dot \phi_x|^2\\
= & -\tfrac{1}{16} (a^2 + 2 |\alpha|^2 + |\gamma|^2) + \tfrac{1}{16} (|\alpha|^2 + a^2) -\tfrac{1}{4} g(\alpha, \beta_\xi) - \tfrac{1}{4}  g(\alpha, \beta_\xi)  - |\dot \phi_x|^2\\
= & -\tfrac{1}{16} |\alpha|^2 -\tfrac{1}{16} |\gamma|^2 - \tfrac{1}{2} g(\alpha, \beta_\xi) - |\dot \phi_x|^2,
\end{align*}
since $g(\alpha,\beta_\xi)=\sum_j\alpha_j\<\xi\wedge e_j\cdot\phi,\dot\phi_x\>=\<\xi\wedge\alpha\cdot\phi,\dot\phi_x\>$. This also gives $|g(\alpha,\beta_\xi)|\leq|\alpha||\dot\phi_x|$ with equality if and only if $\dot \phi_x$ is a multiple of $\xi\wedge\alpha\cdot\phi$. Young's inequality implies that $|\alpha|^2-8|\alpha||\dot\phi_x|+16|\dot\phi_x|^2\geq0$ with equality if and only if $|\alpha| = 4 |\dot \phi_x|$. From there the computation of the kernel is straightforward.
\end{proof}

The kernel of $\sigma_\xi(D_{(g,\phi)}Q)$ reflects the diffeomorphism invariance of the equation. More specifically, consider the principal symbol of the linear operator
\begin{align*}
\lambda^*_{g,\phi}: \Gamma(TM) &\rightarrow \Gamma(\odot^2T^*\!M) \oplus \Gamma(\Sigma_gM)\\
X &\mapsto (\mc L_X g, \tilde {\mc L}^g_X \phi) = (2 \delta_g^* X^\flat, \nabla^g_X\phi-\tfrac{1}{4} dX^\flat \cdot \phi).
\end{align*}
Using $\sigma_\xi (X \mapsto 2 \delta_g^* X^\flat)v=i(\xi \otimes v^\flat + v^\flat \otimes \xi)$ and $\sigma_\xi(X \mapsto\nabla^g_X\phi-\tfrac{1}{4} dX^\flat \cdot \phi)v =-\tfrac{i}{4} \xi \wedge v^\flat \cdot \phi$ we get
\ben
\sigma_\xi (\lambda^*_{g,\phi})v = i ( \xi \otimes v^\flat + v^\flat \otimes \xi, -\tfrac{1}{4} \xi \wedge v^\flat \cdot \phi)
\ee
and thus $\im \sigma_\xi (\lambda^*_{g,\phi})   =  \ker\sigma_\xi(D_{(g,\phi)} Q)$. In fact, if $(g,\phi) \in \mc N$ is a critical point of $\mc E$ and $L_{g,\phi} = D_{(g,\phi)} Q$, then the associated ``deformation complex''

\ben
0 \to \Gamma(TM)\xrightarrow {\lambda_{g,\phi}^*} T_{(g,\phi)} \mc N \xrightarrow{L_{g,\phi}} T_{(g,\phi)} \mc N  \xrightarrow{\lambda_{g,\phi}} \Gamma(TM) \to 0
\ee
is elliptic. This may easily be checked using the Bianchi identity (Corollary~\ref{bianchi}) and the formal self-adjointness of $L_{g,\phi}$.

\medskip

We remedy this situation by adding an additional term to $Q$ in order to break the diffeomorphism invariance. Towards that end, we define a new operator 
\ben
\widetilde Q(\Phi)=Q(\Phi)+\lambda_{\Phi}^*(X(\Phi))
\ee
for a vector field $X(\Phi)$ depending on the $1$-jet of $\Phi=(g,\phi)$. Then $\lambda_{g,\phi}^*(X(\Phi))$ depends on the $2$-jet of $\Phi$ and modifies the highest order terms in $Q$. This procedure is known as DeTurck's trick and was initially applied to give an alternative proof of short-time existence and uniqueness of Ricci flow, cf.~\cite{dt83}. A suitable vector field can be determined as follows. Let $\bar g$ be a background metric. We define the linear operator
\begin{equation}\label{vf}
X_{\bar g}: \mc M \to \Gamma(TM), \quad g \mapsto -2(\delta_{\bar g} g)^\sharp,
\ee
where ${}^\sharp$ is taken with respect to the fixed metric $\bar g$.

\begin{definition}
Let $\bar g$ be a fixed background metric and set
\ben
\widetilde Q_{\bar g}:\mc N\to T\mc N,\quad(g,\phi)\mapsto Q(g,\phi) + \lambda_{g,\phi}^*(X_{\bar g}(g)).
\ee
We call a solution to
\ben
\tfrac{\partial}{\partial t}\tilde\Phi_t=\widetilde Q_{\bar g}(\tilde\Phi_t),\quad\tilde\Phi_0=\bar\Phi
\ee
the {\em gauged spinor flow associated with $\bar g$ and with initial condition} $\bar \Phi=(\bar g,\bar \phi)$ or {\em gauged spinor flow} for short.
\end{definition}

\begin{prop}
Let $\bar g$ be a fixed background metric. Then the gauged spinor flow associated with $\bar g$ is strongly parabolic at $(\bar g, \bar \phi)$ for any $\bar \phi \in \mc N_{\bar g}$.
\end{prop}
\begin{proof}
We need to compute the principal symbol of $D_{(\bar g, \bar \phi)} \widetilde Q_{\bar g}$ and check that the quadratic form
\[
(\dot g_x, \dot \phi_x) \mapsto \bar g\Bigl(\sigma_\xi\bigl(D_{(\bar g,\bar\phi)} \widetilde Q_{\bar g}\bigr) (\dot g_x, \dot \phi_x), (\dot g_x, \dot \phi_x)\Bigr)
\]
is negative definite. The principal symbol of $D_{(\bar g, \bar \phi)}Q$ is known by Proposition~\ref{symbol_Q}. The linearisation of the map $g \mapsto\mc{L}_{X_{\bar g}(g)}g$ at $\bar g$ is given by
\begin{align*}
D_{\bar g}\bigl(g \mapsto \mc{L}_{X_{\bar g}(g)}g\bigr)(\dot g)&=\mc{L}_{X_{\bar g}(\dot g)}\bar g+\!\mbox{ terms of lower order in }\dot g\\
&=-4 \delta_{\bar g}^* \delta_{\bar g}\dot g+\!\mbox{ terms of lower order in }\dot g
\end{align*}
where we have inserted the definition of $X_{\bar g}(\dot g)$ as in \eqref{vf}.
Now for a $1$-form $\zeta \in T_x^*M$ one has $\sigma_\xi(\delta_{\bar g}^\ast)\zeta=\tfrac{i}{2}(\xi\otimes\zeta+\zeta\otimes\xi)$ and $\sigma_\xi(\delta_{\bar g})\dot g_x=-i\dot g_x(\xi,\cdot)$. Hence
\ben
\sigma_\xi (D_{\bar g}(g \mapsto \mc{L}_{X_{\bar g}(g)}g))(\dot g_x)=-2(\xi\otimes\zeta+\zeta\otimes\xi)=-4\xi\odot\zeta,
\ee
where as above $\zeta=\dot g_x(\xi,\cdot)=a\xi+\alpha$. On the other hand, the linearisation of the map $(g,\phi) \mapsto \tilde{\mathcal L}_{X_{\bar g}(g)} \phi$ at $(\bar g,\bar \phi)$ is given by
\begin{align*}
D_{(\bar g, \bar \phi)}\bigl((g,\phi) \mapsto  \tilde{\mathcal L}_{X_{\bar g}(g)} \phi\bigr)(\dot g, \dot \phi) &= \tilde{\mathcal L}_{X_{\bar g}(\dot g)}\bar \phi+\!\mbox{ terms of lower order in }(\dot g,\dot\phi)\\
& = \tfrac{1}{2} d \delta_{\bar g} \dot g \cdot \bar \phi +\!\mbox{ terms of lower order in }(\dot g,\dot\phi)
\end{align*}
and therefore
\[
\sigma_\xi \left(D_{(\bar g, \bar \phi)}\bigl((g,\phi) \mapsto  \tilde{\mathcal L}_{X_{\bar g}(g)} \phi\bigr)\right) (\dot g_x, \dot \phi_x) = \tfrac{1}{2} \xi \wedge \dot g_x(\xi, \cdot) \cdot \bar \phi.
\]
It follows by the previous computations of Corollary~\ref{kerQ} that
\begin{align*}
&\bar g\bigl(\sigma_\xi(D_{(\bar g,\bar\phi)} Q_{\bar g}) (\dot g_x, \dot \phi_x), (\dot g_x, \dot \phi_x)\bigr)\\
=& -\tfrac{1}{16} |\alpha|^2 -\tfrac{1}{16} |\gamma|^2 - \tfrac{1}{2} \bar g(\alpha, \beta_\xi) - |\dot \phi_x|^2-4a^2-4|\alpha|^2 + \tfrac{1}{2} \langle \xi \wedge \dot g_x(\xi , \cdot ) \cdot \bar \phi, \dot \phi_x \rangle\\
=& - 4 a^2 -\tfrac{65}{16} |\alpha|^2  -\tfrac{1}{16} |\gamma|^2 -|\dot \phi_x|^2 \leq  - \tfrac{1}{16} |\dot g_x|^2 - |\dot \phi_x|^2,
\end{align*}
where we used the relation $\bar g(\alpha,\beta_\xi)=\<\xi\wedge\alpha\cdot\bar\phi,\dot\phi_x\>$. This proves the assertion.
\end{proof}

Finally, we are in a position to establish short-time existence of the spinor flow for a given arbitrary initial condition $\bar \Phi = (\bar g, \bar\phi)$. Consider the (unique) solution $\tilde\Phi_t=(\tilde g_t, \tilde \phi_t)$ to the gauged spinor flow equation associated with $\bar g$ subject to the same initial condition $\bar\Phi$. Let~$f_t$ be the $1$-parameter family~$f_t$ in $\diff_0(M)$ determined by the (non-autonomous) ordinary differential equation
\begin{equation}\label{ODE}
\tfrac{d}{dt}f_t=-X_{\bar g}(\tilde g_t)\circ f_t,\quad f_0=\Id_M.
\end{equation}
Then~$f_t$ lifts to a $1$-parameter family $F_t$ in $\widehat \diff_s(M)$ with $F_0=\Id_{\tilde P}$ and we claim that $\Phi_t:=F_t^*\tilde\Phi_t$ is a solution to the spinor flow equation with initial condition $\bar\Phi$. Indeed, we have $\Phi_0=F_0^*\tilde\Phi_0=\bar\Phi$. Furthermore, by Lemma~\ref{pathder} and the equivariance properties established in Proposition~\ref{compat2} we obtain
\begin{align*}
\tfrac{\partial}{\partial t}\Phi_t & = \tfrac{\partial}{\partial t}F_t^*\tilde\Phi_t\\
& = F_t^* \bigl(\widetilde Q_{\bar g}(\tilde\Phi_t) + (\mc L_{-X_{\bar g}(\tilde g_t)}\tilde g_t, \tilde{\mc L}^g_{-X_{\bar g}(\tilde g_t)}\tilde\phi_t)\bigr)\\
& = F_t^* Q(\tilde\Phi_t) = Q(\Phi_t).
\end{align*}
This proves existence.
%
\subsection{Uniqueness}
%
The proof of uniqueness is modelled on the proof in the Ricci flow case given by Hamilton in \cite{ha95}. A similar argument has also been applied in \cite{ww10} for the Dirichlet energy flow on the space of positive $3$-forms on a $7$-manifold. In the existence proof $g_t$ was obtained from $\tilde g_t$ and the solution~$f_t$ of the ODE \eqref{ODE} by setting $g_t = f_t^* \tilde g_t$. Conversely, substituting $\tilde g_t = f_{t*} g_t$ turns \eqref{ODE} into the PDE
\begin{equation}\label{PDE}
\tfrac{d}{dt}f_t=-X_{\bar g}(f_{t*}g_t)\circ f_t,\quad f_0=\Id_M.
\end{equation}
To show uniqueness of the spinor flow we want to use uniqueness of the gauged spinor flow. To that end we pass from a spinor flow solution to a gauged spinor flow solution by solving~\eqref{PDE}. For Riemannian metrics $g$ and $\bar g$ let
\ben
P_{g,\bar g} : C^{\infty}\!(M,M) \rightarrow T C^\infty\!(M,M), \quad f \mapsto -df (X_{f^*\bar g}(g)).
\ee
Note that $P_{g,\bar g}(f) \in \Gamma(f^*TM) = T_f C^\infty\!(M,M)$ so that $P_{g,\bar g}$ may be viewed as a section of the bundle $T C^\infty\!(M,M)\rightarrow C^\infty\!(M,M)$. 
Furthermore, for a diffeomorphism $f$ we recover $-df(X_{f^*\bar g}(g)) = -X_{\bar g}(f_*g)\circ f$ as above.

\begin{lem}
The operator $P_{g,\bar g}$ is a second order, quasilinear differential operator.
\end{lem}

\begin{proof}
The proof essentially amounts to a calculation in local coordinates, cf.\ the proof of Lemma 6.1 in \cite{ww10}. Further details are left to the reader.
\end{proof}

The proof of existence of solutions to~\eqref{PDE} resembles the short-time existence proof of the spinor flow. First, using the exponential map on $M$ with respect to the metric $\bar g$ we may identify maps close to a fixed map $f$ with sections of the bundle $f^*TM$. Furthermore, by parallel translating along radial geodesics, we may assume that $P_{g,\bar g}$ takes values in the same fixed vector bundle $f^*M$. Hence $P_{g,\bar g}$ as well as its linearisation $D_f P_{g,\bar g}$ can be considered as operators $\Gamma(f^*TM) \to \Gamma(f^*TM)$.

\begin{lem}
The linearisation of $P_{\bar g,\bar g}$ at $f = \Id_M$ is given by
\[
D_{\Id_M} P_{\bar g, \bar g}: \Gamma(TM) \to \Gamma(TM), \quad X \mapsto -4 (\delta_{\bar g} \delta_{\bar g}^*X^\flat)^\sharp, 
\]
where ${}^\sharp$ and ${}^\flat$ are taken with respect to $\bar g$.
\end{lem}
\begin{proof}
Let~$f_t$ be the flow of $X \in \Gamma(TM)$. Fix $x \in M$ and let $\tau^{\bar g}_t$ denote the parallel transport induced by the Levi-Civita connection $\nabla^{\bar g}$ along the path $t \mapsto f_t(x)$ with inverse $\tau_{-t}^{\bar g}$. Then
\begin{align*}
\left. \tfrac{d}{dt}\right|_{t=0} \tau_{-t}^{\bar g} (\delta_{\bar g}(f_{t*} \bar g))^\sharp (f_t(x)) &= \left. \tfrac{d}{dt}\right|_{t=0}\tau_{-t}^{\bar g} (\delta_{\bar g}(f_{-t}^* \,\bar g))^\sharp (f_t(x))\\
&= \nabla_X^{\bar g} (\delta_{\bar g}(\bar g))^\sharp(x) - (\delta_{\bar g}(\mc L_X \bar g))^\sharp(x)\\
&= -2 (\delta_{\bar g} \delta_{\bar g}^*X^\flat)^\sharp (x),
\end{align*}
since $\delta_{\bar g} \bar g =0$ and $\mc L_X \bar g = 2 \delta_{\bar g}^*X^\flat$. Taking into account that $X_{\bar g}(g) = -2 (\delta_{\bar g}g)^\sharp$, the asserted formula follows.
\end{proof}

\begin{prop}
Let $g_t$ be a smooth family of Riemannian metrics with $g_0= \bar g$. Then the mapping flow $\frac{d}{dt} f_t =  P_{g_t,\bar g}$ is strongly parabolic at $t=0$ and $f_0=\Id_M$ and hence has a unique short time solution with initial condition $f_0 = \Id_M$.
\end{prop}

\begin{proof}
The symbol of $-4\delta_{\bar g} \delta_{\bar g}^*: \Gamma(T^*\!M) \rightarrow \Gamma(T^*\!M)$ evaluated in $\xi \in T_x^*M$ with $|\xi|=1$ is given by
\[
-4 \sigma_\xi (\delta_{\bar g} \delta_{\bar g}^*)(\zeta) = -2 (\zeta + \bar g(\zeta,\xi) \xi)
\]
for any $\zeta \in T_x^*M$. Hence
\[
-4 \bar g\bigl(\sigma_\xi (\delta_{\bar g} \delta_{\bar g}^*)(\zeta), \zeta\bigr) = -2 \bar g( \zeta, \zeta) -2 \bar g( \zeta, \xi)^2 \leq -2 |\zeta|_{\bar g}^2
\]
for any $\zeta \in T_x^*M$.
\end{proof}

\begin{rem*}
The solution stays a diffeomorphism for short time since $\diff(M)$ is open in $C^\infty\!(M,M)$.
\end{rem*}

The proof of uniqueness can now be finished as follows: Suppose that $\Phi_t=(g_t,\phi_t)$ and $\Phi'_t=(g_t',\phi_t')$ are two spinor flow solutions on $[0,\epsilon]$ with the same initial condition $\bar \Phi = (\bar g, \bar \phi)$. Solving~\eqref{PDE} with $g_t$ and $g'_t$ gives two families of diffeomorphisms~$f_t$ and~$f'_t$ which we assume to be defined on $[0,\epsilon]$. By construction $\tilde\Phi_t=F^*_t\Phi_t$ and $\tilde\Phi'_t=F'^*_t\Phi'_t$ are solutions of the gauged spinor flow equation associated with $\bar g$. Uniqueness of the gauged spinor flow implies that $\tilde\Phi_t=\tilde\Phi'_t$ for all $t \in [0,\epsilon]$. Hence~$f_t$ and~$f'_t$ are solutions of the same ODE \eqref{ODE}. By uniqueness of the solution to ordinary differential equations, we conclude that $f_t=f'_t$, whence $\Phi_t=\Phi'_t$ for all $t \in [0,\epsilon]$.
%
\subsection{Premoduli spaces}
%
We now study the zero locus of $\tilde Q_{\bar g}$ at a critical point $(\bar g,\bar\phi)$ where $\bar g$ is an irreducible metric associated with a torsionfree $G$-structure. We assume $n=\dim M \geq 4$. We claim that
\ben
\tilde Q_{\bar g}^{-1}(0)=\ker X_{\bar g}\cap Q^{-1}(0).
\ee
Indeed, assume that $Q_{\bar g}(\Phi)=Q(\Phi)+\lambda^*_\Phi(X_{\bar g}(\Phi))=0$, the other inclusion being trivial. From the Bianchi identity in Corollary~\ref{bianchi} we deduce $\lambda_\Phi Q(\Phi)+\lambda_\Phi\lambda^*_\Phi(X_{\bar g}(\Phi))=\lambda_\Phi\lambda^*_\Phi(X_{\bar g}(\Phi))=0$. As a consequence, $\llangle\lambda_\Phi\lambda^*_\Phi(X_{\bar g}(\Phi)),X_{\bar g}(\Phi)\rrangle_g=0$ and therefore $\lambda^*_\Phi(X_{\bar g}(\Phi))=0$ using integration by parts. This in turn implies $Q(\Phi)=0$ and that $X$ is Killing for $g$. Since $\bar g$ is irreducible the fundamental group of $M$ must be finite (cf.\ the arguments of Lemma~\ref{irredcrit}), hence the isometry group of $g$ must be finite for $g$ is Ricci-flat. It follows that $X_{\bar g}(g)=0$ as claimed. Now $Q^{-1}(0)$ is a smooth manifold near $\bar\Phi$ by Theorem~\ref{integkerL}. Further, the intersection with $\ker X_{\bar g}$ is transversal. Indeed, $X_{\bar g}(g)=-2(\delta_{\bar g}g)^\sharp=0$ implies that $g$ lies in the {\em Ebin slice}~\cite{eb68} of divergence-free metrics. Since $\tilde Q_{\bar g}$ is elliptic, $\tilde Q^{-1}_{\bar g}(0)$ is smooth near $(g,\phi)$. Strictly speaking we would need to pass to some suitable completion of our function spaces, but we gloss over these technicalities in view of the ellipticity of~$\tilde Q_{\bar g}$. The tangent space is given by
\ben
T_{\bar\Phi}\tilde Q_{\bar g}^{-1}(0)=\ker X_{\bar g}\cap\ker L_{\bar\Phi}=\{(\dot g,\dot\phi)\,|\,\delta_g\dot g=0,\,\nabla^g\dot\phi=0,\,\mc D_g\Psi_{\dot g,\phi}=0\},
\ee
for $\dot g\in\ker D_{\bar g}\scal$, the kernel of the linearisation $g\mapsto\scal^g$. If $\bar g$ is Ricci-flat and if~$\dot g$ is divergence-free, then $\tr_g\dot g=const$~\cite[Theorem 1.174 (e)]{be87}, and the claim follows from~\cite[Proposition 2.2]{wa91}. In fact, we can identify $T_{\bar\Phi}\tilde Q_{\bar g}^{-1}(0)$ with the second cohomology group of the elliptic complex (compare also with the cohomological interpretation of~\cite[Proposition 2.13]{wa91}). In summary, we have proven

\begin{thm}
Under the assumptions of Theorem~\ref{integkerL}, $\tilde Q^{-1}_{\bar g}(0)$ is a slice for the $\widetilde\diff_0(M)$-action on $\mr{Crit}(\mc E)$. In particular, the premoduli space of parallel spinors $\mr{Crit}(\mc E)/\widetilde\diff_0(M)$ is smooth at $(\bar g,\bar\phi)$.
\end{thm}
%
%
\section{The $s$-energy functional}\label{sec.s-energy}
%
%
In this section we study a family of spinorial energy functionals, depending on a parameter $s \in \R$, which generalise the spinorial energy functional. In the seven-dimensional real case these functionals are related to the (generalised) Dirichlet functional(s) of the second and third author~\cite{ww10},~\cite{ww12}.

\begin{definition}
For $s \in \R$ the {\em $s$-energy functional} $\mc{E}_s$ is defined by
\[
\mc{E}_s(\Phi) = \mc{E}(\Phi) + s \cdot \mc{S}(g_\Phi),
\]
where $\mc{S}(g) = \int_M \scal^g \,dv^g$ is the total scalar curvature functional.
\end{definition}

\begin{rem*}
(i) If $\dim M =2$ then by the Gau{\ss}-Bonnet theorem the functionals $\mc E_s$ differ from $\mc E=\mc E_0$ only by a constant. In the following we therefore assume that $\dim M \geq 3$.

\smallskip

(ii) As the functional $\cE_s$ and $\cE$ only differ by a multiple of $\cS$, it is evident that all symmetries of $\cE$ discussed in Section~\ref{symmetries} are symmetries of $\cE_s$ as well.
\end{rem*}

\begin{example*} 
Let $D_g: \Gamma(\Sigma_g M) \rightarrow \Gamma (\Sigma_g M)$ be the Dirac operator. Then using the Weitzenb\"ock formula $D_g^2 \phi = \nabla^{g*}\nabla^g \phi + \tfrac{1}{4} \scal^g \phi$ we obtain
\ben
\frac{1}{2} \int_M |D_g \varphi|^2 \,dv^g = \frac{1}{2} \int_M \langle D_g^2 \varphi, \varphi \rangle \,dv^g = \mc{E}(g,\phi) + \tfrac{1}{8} \mc{S}(g),
\ee   
that is
\ben
\mc{E}_{\tfrac{1}{8}}(\Phi) = \frac{1}{2} \int_M |D_g \varphi|^2 \,dv^g. 
\ee
\end{example*}
%
\subsection{Short-time existence and uniqueness}\label{ex_uni}
%
In this subsection we determine the closed interval $[s_{min},s_{max}]$ of values of $s$ for which the negative gradient flow of the $s$-energy functional $\mc E_s$ is again weakly parabolic. In particular, $0\in[s_{min},s_{max}]$. For values of $s$ in the open interval $(s_{min},s_{max})$ we will show that adding the same diffeomorphism term as before yields a strongly parabolic flow equation, providing a proof of short-time existence and uniqueness for the ``$s$-energy flow''.

\begin{defn}
Let $Q_s:= - \grad \mc E_s = Q -s \cdot \grad \mc S$. We call a solution to
\beq
\tfrac{\partial}{\partial t}\Phi_t=Q_s(\Phi_t),\quad\Phi_0=\Phi\label{s-floweq}
\ee
the {\em $s$-energy flow with initial condition} $\Phi=(g,\phi) \in \mc N$ or {\em $s$-energy flow} for short.
\end{defn}

The first variation of the total scalar curvature $\mc S(g) = \int_M \scal^g \, dv^g$ is well-known. Recall that one has
\ben
\left.\tfrac{d}{dt}\right|_{t=0} \mc S(g + th) = D_g \mc S (h) = \int_M ( -\Ric^g+ \tfrac{1}{2} \scal^g \cdot g, h)_g \, dv^g, 
\ee
see for example~\cite[Proposition 4.17]{be87}, whence
\ben
- \grad \mc S (g) = \Ric^g - \tfrac{1}{2} \scal^g \cdot g.
\ee
Now
\ben
(D_g\Ric)(h) = \tfrac{1}{2} \Delta_L h -\delta_g^*(\delta_g h) - \tfrac{1}{2} \nabla^g d (\tr_g h)
\ee
where $\Delta_Lh = \nabla^{g*}\nabla^g h + \Ric^g \circ h + h \circ \Ric^g - 2 \mathring{R}^g h$ is the Lichnerowicz Laplacian mentioned in Section~\ref{second_var}, and
\ben
(D_g \scal)(h) = \Delta_g \tr_g h + \delta_g(\delta_g h) - (\Ric^g, h)_g, 
\ee
see for example in~\cite[Theorem 1.174]{be87}. It follows that
\begin{align*}
-(D_g \grad \mc S) (h) =& (D_g \Ric) (h) - \tfrac{1}{2} (D_g \scal) (h)\\ 
=&\tfrac{1}{2} \Delta_L h -\delta_g^*(\delta_g h) - \tfrac{1}{2} \nabla^g d (\tr_g h) - \tfrac{1}{2}\Delta_g (\tr_g h) \cdot g -\tfrac{1}{2} \delta_g(\delta_g h) \cdot g\\
& + \tfrac{1}{2} (\Ric^g, h)_g \cdot g - \tfrac{1}{2} \scal^g \cdot h\\
=& \tfrac{1}{2} \nabla^{g*}\nabla^g h -\delta_g^*(\delta_g h) - \tfrac{1}{2} \nabla^g d (\tr_g h) - \tfrac{1}{2}\Delta_g (\tr_g h) \cdot g -\tfrac{1}{2} \delta_g(\delta_g h) \cdot g\\& + \!\mbox{ terms of lower order in }h
\end{align*}
and hence for $\xi \in T_x^*M$ 
\begin{align}\label{symbol_scal}
 - \sigma_\xi (D_g \grad \mc S)(\dot g_x) =& \tfrac{1}{2} |\xi|^2 \dot g_x - \xi \odot \dot g_x(\xi, \cdot)+  \tfrac{1}{2} (\tr_g \dot g_x)  \xi \otimes \xi\notag\\
& - \tfrac{1}{2} |\xi|^2 (\tr_g \dot g_x)  g + \tfrac{1}{2} \dot g_x(\xi,\xi) g.
\end{align}

\begin{lem}
Let $\xi \in T^*_xM$ with $|\xi|=1$. Then with respect to the decomposition $\dot g_x = a \xi \otimes \xi + \xi \otimes \alpha + \alpha \otimes \xi + \gamma$, where $a \in \R$, $g(\alpha,\xi)= 0 $ and $\gamma(\xi, \cdot) = 0$ (see the proof of Corollary \ref{kerQ}), one has
\ben
-\sigma_\xi (D_g \grad \mc S)(\dot g_x) = \tfrac{1}{2} \gamma - \tfrac{1}{2} (\tr_g \gamma) ( g - \xi \otimes \xi).
\ee
In particular,
\ben
-g(\sigma_\xi (D_g \grad \mc S)(\dot g_x), \dot g_x) = \tfrac{1}{2} |\gamma|^2 - \tfrac{1}{2} (\tr_g \gamma)^2.
\ee
\end{lem}
\begin{proof}
Using equation \eqref{symbol_scal} one easily computes
\ben
-\sigma_\xi (D_g \grad \mc S)(\xi \otimes \xi)=(\tfrac{1}{2} -1 + \tfrac{1}{2}) \xi \otimes \xi + (-\tfrac{1}{2} + \tfrac{1}{2}) g =0
\ee
and
\ben
-\sigma_\xi (D_g \grad \mc S)(\xi \odot \alpha) =( \tfrac{1}{2} - \tfrac{1}{2}) \xi \odot \alpha =0.
\ee
Furthermore,
\ben
-\sigma_\xi (D_g \grad \mc S)(\gamma) = \tfrac{1}{2} \gamma + \tfrac{1}{2} (\tr_g \gamma) \xi \otimes \xi - \tfrac{1}{2} (\tr_g \gamma) g,
\ee
which proves the first claim. The second one follows from $g(\gamma, g -\xi \otimes \xi) = \tr_g \gamma$.
\end{proof}

We consider $Q_s = Q - s \cdot \grad \mc S$ as before and put
\ben
\widetilde{Q}_{\bar g, s}(g,\phi):=Q_s(g,\phi) + \lambda_{g,\phi}^*(X_{\bar g}(g))
\ee
with the same vector field $X_{\bar g}(g)=-2 (\delta_{\bar g}g)^\sharp$ as in Section \ref{flow}. As in the case of the ordinary spinor flow we consider the gauged version of the $s$-energy flow equation
\ben
\tfrac{\partial}{\partial t}\tilde\Phi_t=\widetilde Q_{\bar g,s}(\tilde\Phi_t),\quad\tilde\Phi_0=\bar\Phi
\ee
with initial condition $\bar \Phi = (\bar g, \bar \phi) \in \mc N$.

\begin{cor}
Let $n=\dim M\geq3$. 

\smallskip

\noindent(i) $Q_s$ is weakly elliptic at $\Phi$ if and only if $s \in [-\tfrac{1}{8(n-2)},\frac{1}{8}]$.

\smallskip

\noindent(ii) $\widetilde{Q}_{\bar g,s}$ is strongly elliptic at $\bar \Phi = (\bar g, \bar \phi)$ if and only if  $s \in (-\tfrac{1}{8(n-2)},\frac{1}{8})$.
\end{cor}
\begin{proof}
By definition one has $D_\Phi Q_s = D_\Phi Q - s D_g \grad \mc S$ for $\Phi = (g,\phi)$. Hence
\begin{align*}
g( \sigma_\xi (D_\Phi Q_s)(\dot g_x, \dot \phi_x), (\dot g_x, \dot \phi_x)) =& g( \sigma_\xi (D_\Phi Q)(\dot g_x, \dot \phi_x), (\dot g_x, \dot \phi_x))\\
&-s g(\sigma_\xi(D_g \grad \mc S)(\dot g_x), \dot g_x).
\end{align*}
Now, using the notation of Corollary~\ref{kerQ}
\ben
- \tfrac{1}{16} |\gamma|^2 + s (\tfrac{1}{2} |\gamma|^2 - \tfrac{1}{2} (\tr_g\gamma)^2) = 
\begin{cases}
(-\tfrac{1}{16} + \tfrac{s}{2})|\gamma|^2 & : \quad \tr_g \gamma =0\\
(-\tfrac{1}{16} - s ( \tfrac{n-2}{2})) |\gamma|^2 & : \quad \gamma \text{ pure trace}
\end{cases}
\ee
from which the assertions follow using the previous calculations for the individual symbols.
\end{proof}

\begin{rem*}
Note that $\frac{1}{16} \mc D = \mc E_{1/16}$ and $s = \frac{1}{16}$ lies in the range where $\widetilde Q_{\bar g,s}$ is strongly elliptic. In particular, using results from Subsection~\ref{Dirichlet} we recover Theorem 5.1 of~\cite{ww10}. On the other hand, as observed earlier, $ \frac{1}{2} \int_M |D_g \phi|^2 dv^g= \mc E_{1/8}$ and $s=\frac{1}{8} = s_{max}$ is the borderline case where $Q_s$ is still weakly elliptic, but $\widetilde Q_{\bar g,s}$ is not strongly elliptic. The reason for this is that minimisers of $\mc E_{1/8}$ are harmonic spinors. This is a far less restrictive condition on a metric than to admit parallel spinors. In particular, we get an infinite dimensional moduli space of critical points as the following example shows.
\end{rem*}

\begin{example*}
Let $(M,g_0)$ be a $K3$-surface with a Ricci-flat K\"ahler metric. The spinor 
bundle of a $4$-dimensional manifold carries a quaternionic structure, and
the space of parallel spinors on $(M,g_0)$ is a quaternionic vector space 
of dimension $1$. Let~$g_1$ be a non-Ricci-flat metric on $M$. 
Then there are no parallel spinors 
on $(M,g_1)$. We  choose $g_1$ is 
$C^1$-close to~$g_0$. Due to the index theorem, the space of harmonic spinors on $(M,g_1)$ has 
quaternionic dimension $1$. Further, any non-trivial harmonic spinor has no zero. Finally, we also assume that $g_1$ is chosen such
that there are no non-trivial conformal maps from $(M,g_1)$ to itself. 
If we set $\tilde g_1:=|\phi|^{4/(n-1)}g_1$ 
for a harmonic spinor $\phi$, then every harmonic spinor on $(M,\ti g_1)$ has 
constant length, the group of unit quaternions (or equivalently $\SU(2)$) 
acts freely and transitively on the 
harmonic unit spinors over~$\ti g_1$, and up to scaling $\ti g_1$ is the 
only metric in $[g_1]$ with harmonic unit spinors. Thus close to $[g_1]$, 
$\mr{Crit}(\mc E_{1/8})$ is an $\R^+\times \SU(2)$-principal bundle over the 
space of conformal structures, and thus, close to $\ti g_1$, the 
moduli space $\mr{Crit}(\mc E_{1/8})/\widetilde\diff_0(M)$ is smooth and 
infinite-dimensional  which is not possible 
if $\widetilde Q_{\bar g,1/8}$ were 
strongly elliptic. Similar examples exists in other dimensions.
\end{example*}

The same proof as for the spinor flow yields

\begin{thm}
Let $n=\dim M \geq 3$. Then for any $s \in (-\tfrac{1}{8(n-2)},\frac{1}{8})$
the $s$-energy flow equation \eqref{s-floweq} has a unique short-time solution.
\end{thm}
%
\subsection{The generalised Dirichlet energy functional}\label{Dirichlet}
%
From a principal fibre bundle point of view special holonomy metrics correspond to certain orthonormal frame bundles. These arise as extensions of $G$-subbundles to which the Levi-Civita connection reduces. Instead of using spinors one can describe these reduced $G$-bundles by means of forms of special algebraic type. In particular, one can characterise $\Gt$-bundles in terms of {\em positive} $3$-forms, that is, global sections of the fibre bundle associated with the open cone $\Lambda_+\subset\Lambda^3\R^{7*}$, the orbit of $\GL^+_7$ diffeomorphic to $\GL^+_7/\Gt$. More concretely, assume that we have a pair $(g,\phi)\in\mc N$ over a seven dimensional manifold. Take an oriented $g$-orthonormal basis  $E_1,\ldots,E_7$ of $(\R^7,g)$ and consider $\Sigma_7^\R$ the {\em real} spin representation of $\Spin(7)$. The map
\beq\label{bispinor}
\varphi\otimes\psi\in\Sigma_7^\R\otimes\Sigma_7^\R\mapsto\tfrac 14\sum_p\sum_{1\leq i_1<\ldots<i_p\leq n}\langle\varphi,E_{i_1}\cdot\ldots\cdot E_{i_p}\cdot\psi\rangle E_{i_1}\wedge\ldots\wedge E_{i_p}\in\Lambda^*\R^{7*}
\ee
is a $\Spin(7)$-equivariant embedding. We require the factor $1/4$ to turn this map into an isometry, that is, $\langle\cdot\,,\cdot\rangle\otimes\langle\cdot\,,\cdot\rangle=g_{\Lambda^\ast}$ (recall our convention from Section~\ref{spingeo} according to which $E_{i_1}\cdot\ldots\cdot E_{i_p}$ has unit norm). Choosing an explicit representation such that $\vol_g\cdot\phi=\phi$ gives
\beq\label{equiv.inj}
\varphi\otimes\varphi=\frac 14(1+\Omega+\star_g\Omega+\vol_g),
\ee
(cf.\ \cite[Theorem IV.10.19]{lami89} modulo our conventions; also note that a representation with $\vol_g\cdot\phi=-\phi$ would result in different signs of the homogeneous components). By equivariance, all theses algebraic features make sense on a seven-dimensional spin manifold $M^7$ so that we get a well-defined $3$-form $\Omega$ which one checks to be positive. Conversely, such a $3$-form gives rise to a well-defined metric $g_\Omega$ with unit spinor $\phi_\Omega$ (see for instance~\cite{fkms97}).

\medskip

It follows at once that if $\phi$ is parallel (i.e.\ we have holonomy $\Gt$), then $\Omega$ is parallel. In fact consider the twisted Dirac operator $\mc{D}_g:\Gamma(\Sigma_gM\otimes\Sigma_gM)\to \Gamma(\Sigma_gM\otimes\Sigma_gM)$ which is locally defined by $\mc{D}_g(\varphi\otimes\phi)=D_g\varphi\otimes\phi+\sum_ke_k\cdot\varphi\otimes\nabla^g_{e_k}\phi$. Under the map~\eqref{bispinor} $\mc{D}_g$ corresponds to $d+d^\ast$ (see for instance~\cite[Theorem II.5.12]{lami89}). Hence $\Omega$ is closed and coclosed \iff $\mc D_g(\varphi\otimes\varphi)=0$. On the other hand contracting the latter equation with $\langle\cdot\,,\varphi\rangle$ in the second slot we obtain 
\ben
0 = |\phi|^2 D_g\phi + \sum_k\langle\nabla^g_{e_k}\phi,\phi\rangle e_k\cdot\varphi= D_g\phi
\ee 
and thus $\sum_ke_k\cdot\varphi\otimes\nabla^g_{e_k}\varphi=0$. Then contracting from the right with $\langle e_l\cdot\varphi,\cdot\rangle$ gives $\nabla^g_{e_l}\varphi=$ for $l=1,\ldots,7$, whence $\nabla^g\varphi=0$. We thus recover a theorem of Fern\'andez and Gray~\cite{fegr82} which asserts that $d\Omega=0$, $d\star_g\!\Omega=0$ is equivalent to $\nabla^g\Omega=0$.

In~\cite{ww10} and~\cite{ww12}, the second and third author considered the functionals
\ben
\mc D:\Omega_+\to\R,\quad\Omega\mapsto\tfrac{1}{2}\int_M(|d\Omega|^2_{g_\Omega}+|d\star_{g_\Omega}\!\!\Omega|_{g_\Omega}^2)dv^{g_\Omega},
\ee
where the metric $g_\Omega$ is induced by $\Omega$, and
\ben
\mc C:\Omega_+\to\R,\quad\Omega\mapsto\tfrac{1}{2}\int_M|\nabla^{g_\Omega}\Omega|^2_{g_\Omega}dv^{g_\Omega}.
\ee
These belong to the family of {\em generalised Dirichlet energy functionals}. The critical points are the absolute minimisers which correspond, as we have just seen, to holonomy $\Gt$-structures. Further, the negative gradient flow exists for short times and is unique. In fact, these functionals as well as the induced flows can be regarded as a special case of the spinorial flow resp.\ the $s$-energy flow.

\begin{prop}\label{secfunc}
Let $\mc N_\R$ be the bundle of universal real unit spinors associated with the real $\Spin(7)$-representation $\Sigma_7^\R$. Under the isomorphism~\eqref{bispinor} we have
\ben
\mc C=16\mc E_0|_{\mc N_\R}\quad\mbox{and}\quad\mc D=16\mc E_{1/16}|_{\mc N_\R}.
\ee
\end{prop}
\begin{proof}
We have 
\ben
\nabla^g_X\varphi \otimes \varphi + \varphi \otimes \nabla^g_X \varphi=\tfrac 14\big(\nabla_X^g\Omega + \nabla^g_X (\star_\Omega\Omega)\big)
\ee
and $\<\nabla^g_X\varphi, \varphi \> = 0$ for $\varphi$ has unit length, so
\ben
2|\nabla^g_X \varphi|^2=|\nabla^g_X \varphi \otimes \varphi + \varphi \otimes \nabla^g_X \varphi|^2 =\tfrac{1}{16}(|\nabla^g_X \Omega|^2 + |\nabla^g_X (\star_\Omega \Omega)|^2)=\tfrac{1}{8}|\nabla^g_X \Omega|^2,
\ee
and $16|\nabla^g\varphi|^2=|\nabla^g\Omega|^2$ by summing over an orthogonal frame. This  gives the result for $\mc C$. Furthermore, starting again from~\eqref{equiv.inj} and using the correspondence $d+d^*\leftrightarrow\mc D_g$,
\begin{align*}
|(d+d^*)(\Omega+\star_\Omega\Omega)|^2 &=
16\langle D_g\phi\otimes \phi +\sum_k e_k\cdot \phi\otimes \nabla^g_{e_k}\phi,D_g\phi\otimes\phi+\sum_\ell e_\ell\cdot\phi\otimes \nabla^g_{e_\ell}\phi\rangle\\
&= 16(\langle D_g\phi,D_g\phi\rangle|\phi|^2 + \sum_{k,\ell}\langle e_k\cdot\phi,e_\ell\cdot\phi\rangle\langle\nabla^g_{e_k}\phi,\nabla^g_{e_\ell}\phi\rangle)\\
&= 16(|D_g\phi|^2+|\nabla^g\phi|^2).
\end{align*}
Integration yields the desired assertion.
\end{proof}

Since by Corollary~\ref{reality} reality of a spinor is preserved under the flow, we immediately deduce

\begin{cor}
Under the isomorphism~\eqref{bispinor} the negative gradient flows of~$\mc D$ and~$\mc C$ correspond to certain $s$-energy flows.
\end{cor}
%

%
%
%
\end{document}